\documentclass[10pt]{article}

\usepackage{graphicx}
\usepackage{subcaption}

\usepackage[a4paper, left=35mm,right=35mm,top=34mm,bottom=34mm]{geometry}
\usepackage[utf8]{inputenc}
\usepackage[T1]{fontenc}
\usepackage[english]{babel}

\usepackage{enumerate}
\usepackage{graphicx}
\usepackage{hyperref}
\hypersetup{
    colorlinks=true,
    linkcolor=blue,
    filecolor=magenta,      
    urlcolor=cyan,
}

\usepackage{mathtools,amsthm,amssymb,amsfonts}
\usepackage{algorithm}
\usepackage{algpseudocode}
\makeatletter
\def\thm@space@setup{
  \thm@preskip=10pt \thm@postskip=10pt
}
\makeatother
\theoremstyle{plain}
\newtheorem{theorem}{Theorem}
\theoremstyle{plain}
\newtheorem{lemma}[theorem]{Lemma}
\theoremstyle{definition}
\newtheorem{definition}[theorem]{Definition}
\theoremstyle{definition}

\theoremstyle{remark}
\newtheorem*{remark}{Remark}
\theoremstyle{remark}

\usepackage{caption} 
\usepackage{booktabs}
\usepackage{listings}
\usepackage{color}
\definecolor{dkgreen}{rgb}{0,0.6,0}
\definecolor{gray}{rgb}{0.5,0.5,0.5}
\definecolor{mauve}{rgb}{0.58,0,0.82}
\usepackage{enumitem}

\newcommand{\abs}[1]{\left\lvert#1\right\lvert}
\newcommand{\norm}[1]{\left\lVert#1\right\rVert}

\DeclareMathOperator*{\argmin}{arg\,min}

\newcommand{\email}[1]{\protect\href{mailto:#1}{#1}}

\usepackage{xcolor}[2.11]
\colorlet{inlinkcolor}{green!50!black}
\colorlet{exlinkcolor}{red!50!black}
\if@hidelinks
\hypersetup{ hidelinks = true }
\else
\hypersetup{
  colorlinks = true,
  allcolors = inlinkcolor,
  urlcolor = exlinkcolor,
}
\fi

\newenvironment{@abssec}[1]{
        \vspace{.05in}\parindent .0in
        {\upshape\bfseries #1. }\ignorespaces
    }
    {\par\vspace{.1in}}
\renewenvironment{abstract}{\begin{@abssec}{\abstractname}}{\end{@abssec}}
\newenvironment{keywords}{\begin{@abssec}{Keywords}}{\end{@abssec}}

\usepackage{fancyhdr}

\author{
  {\normalsize Qinmeng Zou}\thanks{CentraleSup\'elec, Universit\'e Paris-Saclay, 3 rue Joliot Curie, 91190 Gif-sur-Yvette, France
    (\email{zouqinmeng@gmail.com}, \email{frederic.magoules@hotmail.com}).}
  \and
  {\normalsize Fr\'ed\'eric Magoul\`es\footnotemark[1]}
}
\title{Fast Gradient Methods with Alignment for Symmetric Linear Systems without Using Cauchy Step}
\date{}

\begin{document}
\maketitle
\thispagestyle{fancy}

\begin{abstract}
The performance of gradient methods has been considerably improved by the introduction of delayed parameters.
After two and a half decades, the revealing of second-order information has recently given rise to the Cauchy-based methods with alignment, which reduce asymptotically the search spaces in smaller and smaller dimensions.
They are generally considered as the state of the art of gradient methods.
This paper reveals the spectral properties of minimal gradient and asymptotically optimal steps, and then suggests three fast methods with alignment without using the Cauchy step.
The convergence results are provided, and numerical experiments show that the new methods provide competitive and more stable alternatives to the classical Cauchy-based methods.
In particular, alignment gradient methods present advantages over the Krylov subspace methods in some situations, which makes them attractive in practice.
\end{abstract}


\section{Introduction}
\label{sec:1}

Consider the linear system
\begin{equation}
\label{eq:ls}
Ax=b,
\end{equation}
where $A\in\mathbb{R}^{N\times N}$ is symmetric positive definite (SPD) and $b\in\mathbb{R}^N$.
The solution $x_*$ is the unique global minimizer of strictly convex quadratic function
\begin{equation}
\label{eq:quad}
f(x) = \frac{1}{2}x^\intercal Ax - b^\intercal x.
\end{equation}
The gradient method is of the form
\begin{equation}
\label{eq:x}
x_{n+1} = x_n - \alpha_n g_n,\quad n = 0,\,1,\,\dots,
\end{equation}
where $g_n =  \nabla f(x_n) = Ax_n - b$.
The steepest descent (SD) method, originally proposed in~\cite{Cauchy1847}, defined the steplength by the reciprocal of a Rayleigh quotient of Hessian matrix $A$
\begin{equation}
\label{eq:sd}
\alpha_n^\text{SD} = \frac{g_n^\intercal g_n}{g_n^\intercal Ag_n},
\end{equation}
which is also called Cauchy steplength.
It minimizes the function $f$ or the $A$-norm error and gives theoretically an optimal result in each step
\[
\alpha_n^\text{SD} = \argmin_{\alpha}f(x_n-\alpha g_n) = \argmin_{\alpha}\norm{(I-\alpha A)e_n}_A^2,
\]
where $e_n = x_* - x_n$.
This classical method is known to behave badly in practice.
The directions generated tend to asymptotically alternate between two orthogonal directions leading to a slow convergence~\cite{Akaike1959}.

The first gradient method with retards is the Barzilai-Borwein (BB) method that was originally proposed in \cite{Barzilai1988}.
The BB method is of the form
\[
\alpha_n^\text{BB} = \frac{g_{n-1}^\intercal g_{n-1}}{g_{n-1}^\intercal Ag_{n-1}},
\]
which remedies the convergence issue for ill-conditioned problems by using nonmonotone steplength.
The motivation arose in providing a two-point approximation to the quasi-Newton methods, namely
\[
\alpha_n^\text{BB} = \argmin_{\alpha} \norm{\frac{1}{\alpha}\Delta x-\Delta g}^2,
\]
where $\Delta x=x_n-x_{n-1}$ and $\Delta g=g_n-g_{n-1}$.
Notice that $\alpha_n^\text{BB} = \alpha_{n-1}^\text{SD}$.
There exists a similar method developed by symmetry in~\cite{Barzilai1988}
\[
\alpha_n^\text{BB2} = \frac{g_{n-1}^\intercal Ag_{n-1}}{g_{n-1}^\intercal A^2g_{n-1}},
\]
which imposes as well a quasi-Newton property
\[
\alpha_n^\text{BB2} = \argmin_{\alpha} \norm{\Delta x-\alpha\Delta g}^2.
\]
We remark that $\alpha_n^\text{BB2} = \alpha_{n-1}^\text{MG}$, see Section~\ref{sec:2}.
Practical experience is generally in favor of BB.
The convergence analysis of these methods was given in~\cite{Raydan1993} and~\cite{Dai2002}.
The preconditioned version was established in~\cite{Molina1996}.
A more recent chapter by~\cite{Fletcher2005} discussed the efficiency of BB.
In the years that followed numerous generalizations have appeared, such as alternate methods~\cite{Dai2003,Dai2003b}, cyclic methods~\cite{Friedlander1999,Dai2003,Dai2006b}, adaptive methods~\cite{Zhou2006,Frassoldati2008}, and some general frameworks~\cite{Friedlander1999,Dai2003,Yuan2010}.

There exist several auxiliary steplengths acting as accelerators of other methods.
More precisely, performing occasionally the auxiliary iterative steps could often improve the global convergence.
For example, in order to find the unique minimizer in finitely many iterations in $2$-dimensions, \cite{Yuan2006} proposed a ingenious steplength as follows
\[
\alpha_n^\text{Y} = 2\left(\sqrt{\left(\frac{1}{\alpha_{n-1}^\text{SD}} - \frac{1}{\alpha_n^\text{SD}}\right)^2 + \frac{4\norm{g_n}^2}{\left(\alpha_{n-1}^\text{SD}\right)^2 \norm{g_{n-1}}^2}} + \frac{1}{\alpha_{n-1}^\text{SD}} + \frac{1}{\alpha_n^\text{SD}}\right)^{-1},
\]
which is called Yuan steplength.
Recently, \cite{DeAsmundis2013} proposed a new gradient method that exploits also the spectral properties of SD.
The improvement resorts to a special steplength
\[
\alpha_n^\text{A} = \left(\frac{1}{\alpha_{n-1}^\text{SD}} + \frac{1}{\alpha_n^\text{SD}}\right)^{-1}.
\]
In one direction, these steplengths give rise to some efficient gradient methods.
For example, \cite{Dai2005c} provided several alternate steps, in which we mention here the second variant
\[
\alpha_n^\text{DY} =
\begin{cases}
\alpha_n^\text{SD}, & n\bmod 4 = 0\text{ or }1, \\[2pt]
\alpha_n^\text{Y}, & \text{otherwise},
\end{cases}
\]
which seems to be the most promising variant according to the experiments.
As usual, it does not have a specific name.
Here we call it Dai-Yuan (DY) method \cite{Frassoldati2008}.
A closer examination of Yuan variants revealed that they have a distinguish property called ``decreasing together''~\cite{Dai2005c}.
It means that DY does not sink into any lower subspace spanned by eigenvectors.
Experiments have shown that BB has also such feature.
Important differences come from the fact that BB is a nonmonotone steplength, whereas DY is monotone thus being more stable.

On the other hand, the auxiliary steps lead to gradient methods with alignment such as
\[
\alpha_n^\text{SDA} =
\begin{cases}
\alpha_n^\text{SD}, & n\bmod (d_1+d_2) < d_1, \\[2pt]
\alpha_n^\text{A}, & n\bmod (d_1+d_2) = d_1, \\[2pt]
\alpha_{n-1}^\text{SDA}, & \text{otherwise},
\end{cases}
\]
with $d_1,\,d_2 \ge 1$.
This method is called steepest descent with alignment (SDA).
Here, we choose the version described in~\cite{DeAsmundis2016} without using the switch condition illustrated in~\cite{DeAsmundis2013}, and vary the form while leaving the alignment property unchanged.
Shortly after, they presented another similar method based on Yuan steplength \cite{DeAsmundis2014}, called steepest descent with constant steplength (SDC) which is of the form
\[
\alpha_n^\text{SDC} =
\begin{cases}
\alpha_n^\text{SD}, & n\bmod (d_1+d_2) < d_1, \\[2pt]
\alpha_n^\text{Y}, & n\bmod (d_1+d_2) = d_1, \\[2pt]
\alpha_{n-1}^\text{SDC}, & \text{otherwise},
\end{cases}
\]
with $d_1,\,d_2 \ge 1$.
The main feature of this method is to foster the reduction of gradient components along the eigenvectors of $A$ selectively, and reduce the search space into smaller and smaller dimensions.
The problem tends to have a better and better condition number~\cite{DeAsmundis2014}.
We note that the motivations of SDA and SDC are different according to~\cite{DeAsmundis2013} and~\cite{DeAsmundis2014}.
Since their derivations both involve spectral analysis of Cauchy step, we define here that both of them are regarded as alignment methods.
These two steps seem to be the state of the art of gradient methods and tend to give the best performance among all of these.
Recently, \cite{Gonzaga2016} introduced a general framework of Cauchy steplength with alignment, which breaks the Cauchy cycle by periodically applying some short steplengths.

Despite the good practical performance of alignment methods, all promising formulations are based on the Cauchy steplength in order to ensure the alignment feature.
It is convenient to relax such restriction and jump out of the framework.
In this paper, we address this issue and investigate some gradient methods with the alignment property without Cauchy steplength.
In Section~\ref{sec:2}, we analyze the spectral properties of minimal gradient step.
In Section~\ref{sec:3}, we introduce some new gradient methods by virtue of the basic steplengths and discuss their alignment property.
In Section~\ref{sec:4}, we focus on the convergence analysis of the new methods.
A set of numerical experiments is illustrated in Section~\ref{sec:5} and concluding remarks are drawn in Section~\ref{sec:6}.

\section{Spectral analysis of minimal gradient}
\label{sec:2}

The minimal gradient (MG) method was proposed in \cite{Krasnoselskii1952} which is of the form
\[
\alpha_n^\text{MG} = \frac{g_n^\intercal Ag_n}{g_n^\intercal A^2g_n}.
\]
It minimizes the $2$-norm gradient value
\[
\alpha_n^\text{MG} = \argmin_{\alpha}\norm{g_n - \alpha Ag_n}^2,
\]
where $\norm{\cdot}$ denotes the Euclidean norm of a vector.
Traditionally it does not have a specific name.
From~\cite{Kozjakin1982} we know that it was originally called ``minimal residues''.
However, this term might cause confusion since there exists a Krylov subspace method called MINRES~\cite{Paige1975} which minimizes the norm of the residual through the Lanczos process.
On the other hand, MG is also a special case of the Orthomin($k$) method when $k=1$~\cite{Greenbaum1997}, and thus sometimes called OM~\cite{Ascher2009,vandenDoel2012}.
Here, the name ``minimal gradient'' comes from~\cite{Dai2003b} since it gives an optimal gradient result in each step.

We can assume without loss of generality that
\[
0<\lambda_1\le\dots\le\lambda_N,
\]
where $\{\lambda_1,\,\dots,\,\lambda_N\}$ is the set of eigenvalues of $A$, and $\{v_1,\,\dots,\,v_N\}$ is the set of associated eigenvectors.
Let $\kappa$ be the condition number of $A$ such that
\begin{equation}
\label{eq:kappa}
\kappa = \frac{\lambda_N}{\lambda_1}.
\end{equation}
From \eqref{eq:x} we can deduce that
\begin{equation}
\label{eq:grd}
g_{n+1} = (I-\alpha_n A)g_n.
\end{equation}
There exist real numbers $\zeta_{i,n}$ such that
\begin{equation}
\label{eq:grd2}
g_n = \sum_{i=1}^N \zeta_{i,n}v_i.
\end{equation}
Then, substituting~\eqref{eq:grd2} into~\eqref{eq:grd} implies
\[
\zeta_{i,n+1} = (1-\alpha_n\lambda_i)\zeta_{i,n}.
\]

We know from~\cite{Akaike1959} that the SD method is asymptotically reduced to a search in the $2$-dimensional subspace generated by the two eigenvectors corresponding to the largest and the smallest
eigenvalues of $A$.
Eventually the directions generated tend to zigzag in two orthogonal directions that gives rise to a slow convergence rate.
Such argument was demonstrated by using the following lemma, see~\cite{Akaike1959} and~\cite{Forsythe1968} for more details.
\begin{lemma}
\label{thm:Akaike1959:1,2}
Let $p_0$ be a probability measure attached to $\{\lambda_1,\,\dots,\,\lambda_N\}$ where $p_{i,0}=p_0(\lambda_i)$ and $0<\lambda_1<\dots<\lambda_N$.
Consider a transformation such that
\[
p_{i,n+1} = \frac{\left(\sum_{j=1}^N\lambda_jp_{j,n} - \lambda_i\right)^2}{\sum_{l=1}^N\left(\sum_{j=1}^N\lambda_jp_{j,n} - \lambda_l\right)^2 p_{l,n}}p_{i,n}.
\]
Then,
\[
\lim_{n\rightarrow\infty}p_{i,2n} =
\begin{cases}
p_*, & i=1, \\[2pt]
0, & i\in\{2,\,\dots,\,N-1\}, \\[2pt]
1-p_*, & i=N, \\
\end{cases}
\]
and
\[
\lim_{n\rightarrow\infty}p_{i,2n+1} =
\begin{cases}
1-p_*, & i=1, \\[2pt]
0, & i\in\{2,\,\dots,\,N-1\}, \\[2pt]
p_*, & i=N, \\
\end{cases}
\]
for some $p_*\in(0,\,1)$.
\end{lemma}
We now give our main result on the spectral properties of MG.
These arguments lead to the gradient methods with alignment which shall be described in Section~\ref{sec:3}.
\begin{theorem}
\label{thm:1}
Consider the linear system $Ax=b$ where $A\in\mathbb{R}^{N\times N}$ is SPD and $b\in\mathbb{R}^N$.
Assume that the sequence of solution vectors $\{x_n\}$ is generated by the MG method.
If $0<\lambda_1<\dots<\lambda_N$ and the starting point $x_0$ is such that $\zeta_{1,0} \ne 0$ and $\zeta_{N,0} \ne 0$, then for some constant $c$, the following results hold
\begin{enumerate}[label=(\alph*)]
\item
\begin{equation}
\label{eq:1:a1}
\lim_{n\rightarrow\infty}\frac{\lambda_i\zeta_{i,2n}^2}{\sum_{j=1}^N\lambda_j\zeta_{j,2n}^2} =
\begin{cases}
\frac{1}{1+c^2}, & i=1, \\[2pt]
0, & i\in\{2,\,\dots,\,N-1\}, \\[2pt]
\frac{c^2}{1+c^2}, & i=N,
\end{cases}
\end{equation}
\begin{equation}
\label{eq:1:a2}
\lim_{n\rightarrow\infty}\frac{\lambda_i\zeta_{i,2n+1}^2}{\sum_{j=1}^N\lambda_j\zeta_{j,2n+1}^2} =
\begin{cases}
\frac{c^2}{1+c^2}, & i=1, \\[2pt]
0, & i\in\{2,\,\dots,\,N-1\}, \\[2pt]
\frac{1}{1+c^2}, & i=N;
\end{cases}
\end{equation}
\item
\begin{equation}
\label{eq:1:b1}
\lim_{n\rightarrow\infty}\alpha_{2n}^\textup{MG} = \frac{1+c^2}{\lambda_1(1+c^2\kappa)},
\end{equation}
\begin{equation}
\label{eq:1:b2}
\lim_{n\rightarrow\infty}\alpha_{2n+1}^\textup{MG} = \frac{1+c^2}{\lambda_1(c^2+\kappa)};
\end{equation}
\item
\begin{equation}
\label{eq:1:c}
\lim_{n\rightarrow\infty}\frac{\norm{g_{n+1}}^2}{\norm{g_n}^2} = \frac{c^2(\kappa-1)^2}{(c^2+\kappa)(1+c^2\kappa)};
\end{equation}
\item
\begin{equation}
\label{eq:1:d1}
\lim_{n\rightarrow\infty}\frac{g_{2n+1}^\intercal Ag_{2n+1}}{g_{2n}^\intercal Ag_{2n}} = \frac{c^2(\kappa-1)^2}{(1+c^2\kappa)^2},
\end{equation}
\begin{equation}
\label{eq:1:d2}
\lim_{n\rightarrow\infty}\frac{g_{2n+2}^\intercal Ag_{2n+2}}{g_{2n+1}^\intercal Ag_{2n+1}} = \frac{c^2(\kappa-1)^2}{(c^2+\kappa)^2}.
\end{equation}
\end{enumerate}
\end{theorem}
\begin{proof}
We first prove~\eqref{eq:1:a1} and~\eqref{eq:1:a2}.
We have
\[
\zeta_{i,n+1} = \left(1-\alpha_n^\text{MG}\lambda_i\right)\zeta_{i,n},
\]
Together with~\eqref{eq:grd2}, this implies that
\[
\zeta_{i,n+1} = \left(1-\frac{\sum_{j=1}^N\lambda_j\zeta_{j,n}^2}{\sum_{j=1}^N\lambda_j^2\zeta_{j,n}^2}\lambda_i\right)\zeta_{i,n}.
\]
For any $i$ and $n$, let us write $\hat{p}_{i,n} = \lambda_i\zeta_{i,n}^2$, it follows that
\begin{equation}
\label{eq:phat}
\hat{p}_{i,n+1} = \left(1-\frac{\sum_{j=1}^N\hat{p}_{j,n}}{\sum_{j=1}^N\lambda_j\hat{p}_{j,n}}\lambda_i\right)^2\hat{p}_{i,n}.
\end{equation}
Moreover, we define a probability measure
\begin{equation}
\label{eq:p}
p_{i,n} = \frac{\hat{p}_{i,n}}{\sum_{j=1}^N\hat{p}_{j,n}},
\end{equation}
from which we notice that $\sum_{i=1}^N p_{i,n} = 1$.
Hence,
\[
p_{i,n+1} = \left(\frac{\sum_{j=1}^N\lambda_j p_{j,n}-\lambda_i}{\sum_{j=1}^N\lambda_j p_{j,n}}\right)^2\frac{\hat{p}_{i,n}}{\sum_{l=1}^N \hat{p}_{l,n+1}}.
\]
Notice that $p_*$ in Lemma~\ref{thm:Akaike1959:1,2} can be expressed as $1/(1+c^2)$ without loss of generality.
Substituting \eqref{eq:phat} and applying again~\eqref{eq:p}, it follows that
\[
p_{i,n+1} = \frac{\left(\sum_{j=1}^N\lambda_j p_{j,n}-\lambda_i\right)^2}{\sum_{l=1}^N\left(\sum_{j=1}^N\lambda_j p_{j,n}-\lambda_l\right)^2 p_{l,n}}p_{i,n}.
\]
Along with Lemma~\ref{thm:Akaike1959:1,2} the desired result follows.

For the argument (b), notice that
\[
\alpha_n^\text{MG} = \frac{1}{\sum_{j=1}^N\lambda_j p_{j,n}}.
\]
Since argument (a) has been proved, relations \eqref{eq:1:b1} and \eqref{eq:1:b2} trivially follow by applying \eqref{eq:1:a1} and \eqref{eq:1:a2}.

Then we prove the argument (c).
For any $n$, it follows from \eqref{eq:grd} that
\[
\frac{\norm{g_{n+1}}^2}{\norm{g_n}^2} = \frac{\sum_{j=1}^N\left(1-\alpha_n^\text{MG}\lambda_j\right)^2\zeta_{j,n}^2}{\sum_{j=1}^N\zeta_{j,n}^2}.
\]
Combining \eqref{eq:1:a1} and \eqref{eq:1:b1} implies
\[
\begin{split}
\lim_{n\rightarrow\infty}\frac{\norm{g_{2n+1}}^2}{\norm{g_{2n}}^2} &= \frac{\left(1-\frac{1+c^2}{1+c^2\kappa}\right)^2\lambda_1^{-1}\frac{1}{1+c^2}+\left(1-\frac{(1+c^2)\kappa}{1+c^2\kappa}\right)^2\lambda_N^{-1}\frac{c^2}{1+c^2}}{\lambda_1^{-1}\frac{1}{1+c^2}+\lambda_N^{-1}\frac{c^2}{1+c^2}} \\[2pt]
&= \frac{(\kappa-1)^2c^4\kappa+(\kappa-1)^2c^2}{(c^2+\kappa)(1+c^2\kappa)^2}.
\end{split}
\]
After some simplification, we can obtain \eqref{eq:1:c} when the number of iteration is even in denominator.
In an analogous fashion, combining \eqref{eq:1:a2} and \eqref{eq:1:b2} yields
\[
\begin{split}
\lim_{n\rightarrow\infty}\frac{\norm{g_{2n+2}}^2}{\norm{g_{2n+1}}^2} &= \frac{\left(1-\frac{1+c^2}{c^2+\kappa}\right)^2\lambda_1^{-1}\frac{c^2}{1+c^2}+\left(1-\frac{(1+c^2)\kappa}{c^2+\kappa}\right)^2\lambda_N^{-1}\frac{1}{1+c^2}}{\lambda_1^{-1}\frac{c^2}{1+c^2}+\lambda_N^{-1}\frac{1}{1+c^2}} \\[2pt]
&= \frac{(\kappa-1)^2c^2\kappa+(\kappa-1)^2c^4}{(c^2+\kappa)^2(1+c^2\kappa)}.
\end{split}
\]
One finds that the final result of the odd case converges also to the same limit, which is the desired conclusion.

Finally, for the argument (d), we can similarly combine \eqref{eq:1:a1} and \eqref{eq:1:b1}, which implies
\[
\begin{split}
\lim_{n\rightarrow\infty}\frac{g_{2n+1}^\intercal Ag_{2n+1}}{g_{2n}^\intercal Ag_{2n}} &= \left(1-\frac{1+c^2}{\lambda_1(1+c^2\kappa)}\lambda_1\right)^2\frac{1}{1+c^2} + \left(1-\frac{1+c^2}{\lambda_1(1+c^2\kappa)}\lambda_N\right)^2\frac{c^2}{1+c^2} \\[2pt]
&= \frac{c^4(\kappa-1)^2+c^2(\kappa-1)^2}{(1+c^2\kappa)^2(1+c^2)}.
\end{split}
\]
Repeating this process for another case by using \eqref{eq:1:a2} and \eqref{eq:1:b2} yields
\[
\begin{split}
\lim_{n\rightarrow\infty}\frac{g_{2n+2}^\intercal Ag_{2n+2}}{g_{2n+1}^\intercal Ag_{2n+1}} &= \left(1-\frac{1+c^2}{\lambda_1(c^2+\kappa)}\lambda_1\right)^2\frac{c^2}{1+c^2} + \left(1-\frac{1+c^2}{\lambda_1(c^2+\kappa)}\lambda_N\right)^2\frac{1}{1+c^2} \\[2pt]
&= \frac{c^2(\kappa-1)^2+c^4(\kappa-1)^2}{(c^2+\kappa)^2(1+c^2)}.
\end{split}
\]
After some simplification, we can obtain \eqref{eq:1:d1} and \eqref{eq:1:d2}.
This completes our proof.
\end{proof}
\begin{remark}
The assumption used in Theorem~\ref{thm:1} is not restrictive since if there exist some repeated eigenvalues, then we can choose the corresponding eigenvectors so that the superfluous ones vanish~\cite{Fletcher2005}.
Moreover, if $\zeta_{1,0}$ or $\zeta_{N,0}$ equals zero, then the second condition can be simply replaced by the components involving inner indices without changing the results discussed later on.
\end{remark}
Note that argument~(a) in Theorem~\ref{thm:1} has been proved in~\cite{Pronzato2006} for a framework called $P$-gradient algorithms, while results~(b) to~(d) for the MG method have not appeared in any literature.
(b) shows that MG has also the zigzag behavior, namely, $\alpha_n$ alternates between two directions.
The implications for Theorem~\ref{thm:1} shall be seen later in Section~\ref{sec:3}.
For now, we give the asymptotic behavior of the quadratic function $f$ for completeness.
\begin{theorem}
\label{thm:2}
Under the assumptions of Theorem~\ref{thm:1}, the following results hold
\begin{equation}
\label{eq:2:1}
\lim_{n\rightarrow\infty}\frac{f(x_{2n+1})-f(x_*)}{f(x_{2n})-f(x_*)} = \frac{c^2(1+c^2\kappa^2)(\kappa-1)^2}{(c^2+\kappa^2)(1+c^2\kappa)^2},
\end{equation}
\begin{equation}
\label{eq:2:2}
\lim_{n\rightarrow\infty}\frac{f(x_{2n+2})-f(x_*)}{f(x_{2n+1})-f(x_*)} = \frac{c^2(c^2+\kappa^2)(\kappa-1)^2}{(1+c^2\kappa^2)(c^2+\kappa)^2},
\end{equation}
and
\begin{equation}
\label{eq:2:3}
\lim_{n\rightarrow\infty}\frac{f(x_{2n+2})-f(x_*)}{f(x_{2n})-f(x_*)} = \lim_{n\rightarrow\infty}\frac{\norm{g_{n+1}}^4}{\norm{g_n}^4}.
\end{equation}
\end{theorem}
\begin{proof}
For any $n$, it follows from~\eqref{eq:quad} that
\[
\frac{f(x_{n+1})-f(x_*)}{f(x_n)-f(x_*)} = 1 + \frac{\left(g_n^\intercal Ag_n\right)^3}{\left(g_n^\intercal A^{-1}g_n\right)\left(g_n^\intercal A^2g_n\right)^2} - \frac{2\left(g_n^\intercal Ag_n\right)\left(g_n^\intercal g_n\right)}{\left(g_n^\intercal A^2g_n\right)\left(g_n^\intercal A^{-1}g_n\right)}.
\]
Let us write $p_{i,n}$ as defined in~\eqref{eq:phat} and~\eqref{eq:p}, in which case we obtain
\[
\frac{f(x_{n+1})-f(x_*)}{f(x_n)-f(x_*)} = 1 + \frac{1}{\left(\sum_{j=1}^N\lambda_j^{-2}p_{j,n}\right)\left(\sum_{j=1}^N\lambda_j p_{j,n}\right)^2} - \frac{2\sum_{j=1}^N\lambda_j^{-1} p_{j,n}}{\left(\sum_{j=1}^N\lambda_j^{-2}p_{j,n}\right)\left(\sum_{j=1}^N\lambda_j p_{j,n}\right)}.
\]
If $n$ is an even number, from~\eqref{eq:1:a1}, one finds that
\[
\begin{split}
\lim_{n\rightarrow\infty}\frac{f(x_{n+1})-f(x_*)}{f(x_n)-f(x_*)} &= 1 + \frac{1}{\left(\frac{\kappa^2+c^2}{1+c^2}\right)\left(\frac{1+\kappa c^2}{\kappa(1+c^2)}\right)^2} - \frac{2\left(\frac{\kappa+c^2}{1+c^2}\right)}{\left(\frac{\kappa^2+c^2}{1+c^2}\right)\left(\frac{1+\kappa c^2}{\kappa(1+c^2)}\right)} \\[2pt]
&= \frac{\kappa^4c^4-2\kappa^3c^4+\kappa^2c^4+\kappa^2c^2-2\kappa c^2+c^2}{(c^2+\kappa^2)(1+c^2\kappa)^2}.
\end{split}
\]
Notice that
\[
\kappa^4c^4-2\kappa^3c^4+\kappa^2c^4+\kappa^2c^2-2\kappa c^2+c^2 = c^2(1+c^2\kappa^2)(\kappa-1)^2,
\]
which yields the first equation.
Similarly, if $n$ is an odd number, it follows that
\[
\begin{split}
\lim_{n\rightarrow\infty}\frac{f(x_{n+1})-f(x_*)}{f(x_n)-f(x_*)} &= 1 + \frac{1}{\left(\frac{\kappa^2c^2+1}{1+c^2}\right)\left(\frac{c^2+\kappa}{\kappa(1+c^2)}\right)^2} - \frac{2\left(\frac{\kappa c^2+1}{1+c^2}\right)}{\left(\frac{\kappa^2c^2+1}{1+c^2}\right)\left(\frac{c^2+\kappa}{\kappa(1+c^2)}\right)} \\[2pt]
&= \frac{\kappa^2c^4-2\kappa c^4+c^4+\kappa^4c^2-2\kappa^3c^2+\kappa^2c^2}{(c^2\kappa^2+1)(c^2+\kappa)^2}.
\end{split}
\]
The numerator can be merged as follows
\[
\kappa^2c^4-2\kappa c^4+c^4+\kappa^4c^2-2\kappa^3c^2+\kappa^2c^2 = c^2(c^2+\kappa^2)(\kappa-1)^2,
\]
which yields the second result.
Finally, \eqref{eq:2:3} follows immediately by combining~\eqref{eq:2:1}, \eqref{eq:2:2} and~\eqref{eq:1:c}.
This completes our proof.
\end{proof}

\section{New alignment methods without Cauchy steplength}
\label{sec:3}

As far as we know, all existing gradient methods with alignment are based on the Cauchy steplength.
After a further rearrangement of steps, \cite{Gonzaga2016} concludes that one could break the Cauchy cycle by periodically applying some short steplengths to accelerate the convergence of gradient methods.
We show here that such condition is not necessary and several methods that potentially possess the same feature without Cauchy step can be derived.

\cite{DeAsmundis2013} observed that a constant equal to~$1/(\lambda_1+\lambda_N)$ could lead to alignment property.
Here we extend it to a more general case.
\begin{theorem}
\label{thm:3}
Consider the linear system \eqref{eq:ls} and the gradient method \eqref{eq:x} with a positive constant steplength $\hat{\alpha}$ such that
\begin{equation}
\label{eq:const}
\hat{\alpha} \le \frac{2}{\lambda_1+\lambda_N}
\end{equation}
being used to solve \eqref{eq:ls}.
Then the sequence $\{x_n\}$ converges to $x_*$ for any starting point $x_0$.
Moreover, if equality holds, then
\begin{equation}
\label{eq:alm1}
\lim_{n\rightarrow\infty}\frac{\zeta_{i,n}}{\zeta_{1,n}} =
\begin{cases}
0, & i = 2,\,3,\,\dots,\,N-1, \\[2pt]
\frac{\zeta_{N,0}}{\zeta_{1,0}}(-1)^n, & i = N;
\end{cases}
\end{equation}
otherwise,
\begin{equation}
\label{eq:alm2}
\lim_{n\rightarrow\infty}\frac{\zeta_{i,n}}{\zeta_{1,n}} = 0,\quad i = 2,\,3,\,\dots,\,N.
\end{equation}
\end{theorem}
\begin{proof}
We have
\[
\hat{\alpha} \le \frac{2}{\lambda_1+\lambda_N} < \frac{2}{\lambda_N} \le 2\alpha_n^\text{SD}.
\]
By~\cite{Raydan2002}, it is easy to deduce that the sequence $\{x_n\}$ converges to $x_*$ with a steplength $\alpha < 2\alpha_n^\text{SD}$.
Hence, the first statement holds.
One finds that
\[
\lim_{n\rightarrow\infty}\frac{\zeta_{i,n}}{\zeta_{1,n}} = \frac{\zeta_{i,0}}{\zeta_{1,0}} \lim_{n\rightarrow\infty} \left(\frac{1-\hat{\alpha}\lambda_i}{1-\hat{\alpha}\lambda_1}\right)^n.
\]
Let
\[
\varphi_i = \frac{1-\hat{\alpha}\lambda_i}{1-\hat{\alpha}\lambda_1}.
\]
For~\eqref{eq:alm2} to be satisfied, one needs to impose the condition $\abs{\varphi_i} < 1$ for all $i = 2,\,3,\,\dots,\,N$, which yields
\[
(\lambda_i+\lambda_1)\hat{\alpha} < 2,\quad (\lambda_i-\lambda_1)\hat{\alpha} > 0.
\]
The second one is obviously satisfied, while the first one leads to
\[
\hat{\alpha} < \frac{2}{\lambda_1+\lambda_N}.
\]
If equality holds, then
\[
\varphi_i = \frac{\lambda_N+\lambda_1-2\lambda_i}{\lambda_N-\lambda_1},
\]
It is clear that $\varphi_N = -1$.
Then the second statement trivially follows, which completes the proof.
\end{proof}

Note that $i=1$ leads to the trivial case $\varphi_1=1$, and thus the limit in both~\eqref{eq:alm1} and~\eqref{eq:alm2} equals~$1$.
From Theorem~\ref{thm:3} we find that condition~\eqref{eq:const} has a twofold effect:
driving the alignment property when strict partial order holds, as shown in~\eqref{eq:alm2}, and forcing the search into a two-dimensional space in the equal case, as shown in~\eqref{eq:alm1}.
It means that if there exist some steps asymptotically making the equality of~\eqref{eq:alm1} attainable, then it has similar tendency with the SD method, namely, alternating between two orthogonal directions.
On the other hand, we can add a fractional factor to periodically break the cycle.
This asymptotically yields a constant steplength strictly smaller than $2/(\lambda_1+\lambda_N)$, leading to alignment process in the subsequent several iterations according to~\eqref{eq:alm2}.

Recall that~\cite{Dai2006} proposed a gradient method of the form
\[
\alpha_n^\text{AO} = \frac{\norm{g_n}}{\norm{Ag_n}}.
\]
It asymptotically converges to the optimal steplength
\[
\lim_{n\rightarrow\infty}\alpha_n^\text{AO} = \alpha^\text{OPT} = \frac{2}{\lambda_1+\lambda_N},
\]
which minimizes the coefficient matrix
\[
\alpha^\text{OPT} = \argmin_{\alpha}\norm{I - \alpha A}.
\]
Thus we call it asymptotically optimal (AO) method.
Notice that the following relationship holds
\begin{equation}
\label{eq:rls}
\alpha_n^\text{MG} \le \alpha_n^\text{AO} \le \alpha_n^\text{SD},
\end{equation}
which can be easily proved by the Cauchy-Schwarz inequality
\[
\frac{g_n^\intercal Ag_n}{g_n^\intercal A^2g_n} \le \frac{\norm{g_n} \norm{Ag_n}}{\norm{Ag_n}^2} = \frac{\norm{g_n}^2}{\norm{Ag_n} \norm{g_n}} \le \frac{g_n^\intercal g_n}{g_n^\intercal Ag_n}.
\]
It is known that AO generates monotone curve and often leads to slow convergence.

We observe that the limit of AO satisfies condition~\eqref{eq:alm2} and may potentially be improved by a cyclic breaking.
For example, we can choose a shorter one to constantly align the gradient vector to the one-dimensional space spanned by $v_1$.
Let $\tilde{\alpha}_n=\theta\alpha_n^\text{AO}$ where~$0<\theta<1$.
It follows that
\[
\lim_{n\rightarrow\infty}\tilde{\alpha}_n < \frac{2}{\lambda_1+\lambda_N}.
\]
From Theorem~\ref{thm:3}, we observe that $\tilde{\alpha}_n$ can asymptotically trigger the alignment behavior.
Hence, we can write a new gradient method called AO with alignment (AOA) as follows
\begin{equation}
\label{eq:aoa}
\alpha_n^\text{AOA} =
\begin{cases}
\alpha_n^\text{AO}, & n\bmod (d_1+d_2) < d_1, \\[2pt]
\tilde{\alpha}_n, & n\bmod (d_1+d_2) = d_1, \\[2pt]
\alpha_{n-1}^\text{AOA}, & \text{otherwise},
\end{cases}
\end{equation}
with $d_1,\,d_2 \ge 1$.
Important differences between SDA and AOA come from the fact that the Cauchy step in SDA zigzags itself in two orthogonal directions, while the AO step in AOA converges to a constant and the constant leads later to the same feature.
\\

On the other hand, since the spectral properties of MG have been studied in Section~\ref{sec:2}, we are now prepared to propose our new methods based on them.
We first give some notations
\[
\alpha_n^\text{A2} = \left(\frac{1}{\alpha_{n-1}^\text{MG}} + \frac{1}{\alpha_n^\text{MG}}\right)^{-1},
\]
\[
\alpha_n^\text{Y2} = 2\left(\sqrt{\left(\frac{1}{\alpha_{n-1}^\text{MG}} - \frac{1}{\alpha_n^\text{MG}}\right)^2 + \frac{4g_n^\intercal Ag_n}{\left(\alpha_{n-1}^\text{MG}\right)^2 g_{n-1}^\intercal Ag_{n-1}}} + \frac{1}{\alpha_{n-1}^\text{MG}} + \frac{1}{\alpha_n^\text{MG}}\right)^{-1}.
\]
Note that Y2 has been proposed in~\cite{Dai2005c} as a component of the 2-dimensional finite termination method.
\begin{theorem}
\label{thm:4}
Consider the linear system $Ax=b$ where $A\in\mathbb{R}^{N\times N}$ is SPD and $b\in\mathbb{R}^N$.
Assume that the sequence of solution vectors $\{x_n\}$ is generated by the MG method.
If $0<\lambda_1<\dots<\lambda_N$ and the starting point $x_0$ is such that $\zeta_{1,0} \ne 0$ and $\zeta_{N,0} \ne 0$, then the following results hold
\begin{equation}
\label{eq:4:1}
\lim_{n\rightarrow\infty}\alpha_n^\textup{A2} = \frac{1}{\lambda_1+\lambda_N},
\end{equation}
\begin{equation}
\label{eq:4:2}
\lim_{n\rightarrow\infty}\alpha_n^\textup{Y2} = \frac{1}{\lambda_N},
\end{equation}
and
\begin{equation}
\label{eq:4:3}
\lim_{n\rightarrow\infty}\left(\frac{1}{\alpha_{n-1}^\textup{MG}\alpha_n^\textup{MG}} - \frac{g_n^\intercal Ag_n}{\left(\alpha_{n-1}^\textup{MG}\right)^2 g_{n-1}^\intercal Ag_{n-1}}\right) = \lambda_1\lambda_N.
\end{equation}
\end{theorem}
\begin{proof}
The first conclusion follows immediately by combining \eqref{eq:1:b1} and \eqref{eq:1:b2}.
For the second argument, we have
\[
\alpha_n^\text{Y2} = 2\left(\sqrt{\left(\alpha_n^\text{A2}\right)^{-2} - \frac{4}{\alpha_{n-1}^\text{MG}\alpha_n^\text{MG}} + \frac{4g_n^\intercal Ag_n}{\left(\alpha_{n-1}^\text{MG}\right)^2 g_{n-1}^\intercal Ag_{n-1}}} + \left(\alpha_n^\text{A2}\right)^{-1}\right)^{-1}.
\]
By combining \eqref{eq:1:b1}, \eqref{eq:1:b2}, \eqref{eq:1:d1} and \eqref{eq:1:d2}, it follows that
\[
\lim_{n\rightarrow\infty}\frac{g_{2n+2}^\intercal Ag_{2n+2}}{\left(\alpha_{2n+1}^\text{MG}\right)^2 g_{2n+1}^\intercal Ag_{2n+1}} = \lim_{n\rightarrow\infty}\frac{g_{2n+1}^\intercal Ag_{2n+1}}{\left(\alpha_{2n}^\text{MG}\right)^2 g_{2n}^\intercal Ag_{2n}} = \frac{\lambda_1^2c^2(\kappa-1)^2}{(1+c^2)^2}.
\]
Hence, one can see that
\[
\lim_{n\rightarrow\infty}\left(\frac{1}{\alpha_{n-1}^\text{MG}\alpha_n^\text{MG}} - \frac{g_n^\intercal Ag_n}{\left(\alpha_{n-1}^\text{MG}\right)^2 g_{n-1}^\intercal Ag_{n-1}}\right) = \frac{\lambda_1^2(1+c^2\kappa)(c^2+\kappa)}{(1+c^2)^2} - \frac{\lambda_1^2c^2(\kappa-1)^2}{(1+c^2)^2},
\]
which implies the second conclusion after some simplification.
Further, along with \eqref{eq:4:1}, we have
\[
\lim_{n\rightarrow\infty}\alpha_n^\text{Y2} = 2\left(\sqrt{(\lambda_1+\lambda_N)^2-4\lambda_1\lambda_N}+\lambda_1+\lambda_N\right)^{-1} = \frac{1}{\lambda_N}.
\]
This completes our proof.
\end{proof}

One may conclude from Theorem~\ref{thm:4} that A2 and Y2 are similar to the auxiliary steplengths discussed in~\cite{DeAsmundis2013} and~\cite{DeAsmundis2014}.
However, since MG has shorter steplength than SD, we expect that the former might be more smoother than the latter.
After a substitution of labels, we are able to define MG with alignment (MGA) and MG with constant steplength (MGC) as follows
\begin{equation}
\label{eq:mga}
\alpha_n^\text{MGA} =
\begin{cases}
\alpha_n^\text{MG}, & n\bmod (d_1+d_2) < d_1, \\[2pt]
\alpha_n^\text{A2}, & n\bmod (d_1+d_2) = d_1, \\[2pt]
\alpha_{n-1}^\text{MGA}, & \text{otherwise},
\end{cases}
\end{equation}
\begin{equation}
\label{eq:mgc}
\alpha_n^\text{MGC} =
\begin{cases}
\alpha_n^\text{MG}, & n\bmod (d_1+d_2) < d_1, \\[2pt]
\alpha_n^\text{Y2}, & n\bmod (d_1+d_2) = d_1, \\[2pt]
\alpha_{n-1}^\text{MGC}, & \text{otherwise},
\end{cases}
\end{equation}
with $d_1,\,d_2 \ge 1$.
Recall that the motivation in~\cite{DeAsmundis2013} is to align the algorithm search into the one-dimensional space spanned by $v_1$, which can be summarized by Theorem~\ref{thm:3}.
On the other hand, the strategy in~\cite{DeAsmundis2014} is to foster a special steplength towards the inverse of the largest eigenvalue for which the gradient element has not vanished.
One could easily conclude from Theorem~\ref{thm:3} that $\hat{\alpha}=1/\lambda_N$ satisfies also the former motivation, while $\hat{\alpha}=1/(\lambda_1+\lambda_N)$ may not satisfy the latter one which depends on the relative magnitude of $\lambda_1$.
This may explain the superiority of SDC compared to SDA, and we will see later that this argument remains true for MGA and MGC.

The spectral properties that have been discussed above can be generalized to other basic steplengths of the form
\[
\alpha_n = \frac{g_n^\intercal A^{\rho}g_n}{g_n^\intercal A^{\rho+1}g_n},
\]
with $\rho\ge0$.
Nonetheless, formulations other than SD and MG are not viewed as promising since extra sparse matrix-vector multiplication is required, which often give similar convergence results but at tremendous computational cost.

\section{Convergence analysis}
\label{sec:4}

For the convergence analysis of the aforementioned methods, recall that a convergence framework has been established in~\cite{Dai2003} which requires a tool called Property A.
\begin{definition}[Property A]
Assume that $A = \text{diag}(\lambda_1,\,\dots,\,\lambda_N)$ and $\lambda_1=1$.
Let $g_{i,n}$ be the $i$th component of $g_n$ and 
\[
G(n,\mu) = \sum_{i=1}^\mu g_{i,n}^2.
\]
If $\exists m_0\in\mathbb{N}$, $\exists c_1,c_2>0$, such that
$\forall\mu\in\{1,\,\dots,\,N-1\}$, $\forall\varepsilon>0$, $\forall j\in\{0,\,\dots,\,\min\{n,m_0\}\}$,
\begin{enumerate}
\item $\lambda_1\le\alpha_n^{-1}\le c_1$;
\item if $G(n-j,\mu)\le\varepsilon$ and $g_{\mu+1,n-j}^2\ge c_2\varepsilon$, then $\alpha_n^{-1}\ge\frac{2}{3}\lambda_{\mu+1}$,
\end{enumerate}
then the steplength $\alpha_n$ has Property A.
\end{definition}
\begin{remark}
The assumption used in the above definition seems to be quite strict in practice.
For the theoretical analysis, however, we could simply add an orthogonal transformation that transforms $A$ to a diagonal matrix of eigenvalues.
Additionally, if $\lambda_1 \neq 1$, we could add a factor $1/\lambda_1$ to the matrix without changing the convergence property.
Hence, we make this assumption in some situations for the sake of convergence analysis exclusively.
\end{remark}
A general convergence result can therefore be deduced.
We state the lemma without proof, see~\cite{Dai2003} for more details.
\begin{lemma}
\label{thm:Dai2003:4.1}
Consider the linear system $Ax=b$ with $A=\text{diag}(\lambda_1,\,\dots,\,\lambda_N)$ and $\lambda_1=1$.
Consider the gradient method \eqref{eq:x} where the steplength $\alpha_n$ has Property A.
Then the sequence $\{\norm{g_n}\}$ converges to zero $R$-linearly.
\end{lemma}

Inspired by the pioneering work of BB, \cite{Friedlander1999} proposed a general framework called gradient method with retards (GMR), but AO can not be directly formalized by such framework.
Given $m$ a positive integer, let $\bar{n} = \max\{0,\,n-m\}$.
A generalization of GMR~\cite{Dai2003} can be defined as follows
\begin{equation}
\label{eq:dgmr}
\alpha_n^\text{DGMR} = \left(\frac{g_{\tau(n)}^\intercal A^{\rho(n)}g_{\tau(n)}}{g_{\tau(n)}^\intercal A^{\rho(n)+\upsilon}g_{\tau(n)}}\right)^\frac{1}{\upsilon},
\end{equation}
with
\[
\tau(n)\in\left\{\bar{n},\,\bar{n}+1,\,\dots,\,n-1,\,n\right\},\quad \rho(n)\in\left\{q_1,\,\dots,\,q_m\right\},\quad q_j \ge 0,\quad \upsilon > 0.
\]
Here, we call it Dai's generalization of GMR (DGMR).
After a further selection of parameters $\rho(n)$ and $\tau(n)$, we observe that SD, MG, BB are both special cases of this framework, as well as many other alternate and cyclic gradient methods~\cite{Dai2003,Dai2003b,Dai2006b}.
The convergence of DGMR is summarized in Theorem~\ref{thm:Dai2003:P408}.
\cite{Dai2003} stated this result without proof.
Here, a complete proof is provided and shall also be exploited later by other theorems.
\begin{theorem}
\label{thm:Dai2003:P408}
Consider the gradient method \eqref{eq:x} with steplength \eqref{eq:dgmr} being used to solve the linear system \eqref{eq:ls}.
Then the sequence $\{x_n\}$ converges to $x_*$ for any starting point $x_0$.
\end{theorem}
\begin{proof}
Since gradient methods are invariant under orthogonal transformations, we assume without loss of generality that $A=\text{diag}(\lambda_1,\,\dots,\,\lambda_N)$ and $\lambda_1=1$.
Let
\[
R(A,u) = \frac{u^\intercal Au}{u^\intercal u}
\]
be the Rayleigh quotient for non-zero vector $u$.
Let
\[
u_1=A^{(\rho(n)+\upsilon-1)/2}g_{\tau(n)},
\]
it follows that
\[
\alpha_n^\text{DGMR} = \left(\frac{1}{R(A,u_1)}\cdot\frac{g_{\tau(n)}^\intercal A^{\rho(n)}g_{\tau(n)}}{g_{\tau(n)}^\intercal A^{\rho(n)+\upsilon-1}g_{\tau(n)}}\right)^\frac{1}{\upsilon} \le \left(\frac{1}{\lambda_1} \cdot \frac{g_{\tau(n)}^\intercal A^{\rho(n)}g_{\tau(n)}}{g_{\tau(n)}^\intercal A^{\rho(n)+\upsilon-1}g_{\tau(n)}}\right)^\frac{1}{\upsilon}.
\]
Applying this result recursively, one has
\[
\alpha_n^\text{DGMR} \le \left(\frac{1}{\lambda_1^{\upsilon-1}} \cdot \frac{g_{\tau(n)}^\intercal A^{\rho(n)}g_{\tau(n)}}{g_{\tau(n)}^\intercal A^{\rho(n)+1}g_{\tau(n)}}\right)^\frac{1}{\upsilon} \le \frac{1}{\lambda_1}.
\]
It follows from the similar deduction that
\[
\alpha_n^\text{DGMR} \ge \frac{1}{\lambda_N}.
\]
Thus we can choose $c_1 = \lambda_N$, and then the first relationship of Property A trivially follows.
For the second one, we choose $c_2$ of the form
\begin{equation}
\label{eq:c2}
c_2 = \frac{\left(\frac{2}{3}\right)^\upsilon}{1-\left(\frac{2}{3}\right)^\upsilon}\lambda_\mu^{\bar{q}},
\end{equation}
where $\bar{q} = \max_{i\in[1,m_0]}q_i$.
Let $m_0=m$.
For all $\mu\in\{1,\,\dots,\,N-1\}$ and $j\in\{0,\,\dots,\,\min\{n,m_0\}\}$, one obtains that
\[
\begin{split}
\left(\alpha_n^\text{DGMR}\right)^{-1} & = \left(\frac{\sum_{i=1}^N g_{i,\tau(n)}^2 \lambda_i^{\rho(n)+\upsilon}}{\sum_{i=1}^\mu g_{i,\tau(n)}^2 \lambda_i^{\rho(n)} + \sum_{i=\mu+1}^N g_{i,\tau(n)}^2 \lambda_i^{\rho(n)}}\right)^\frac{1}{\upsilon} \\[2pt]
& \ge \left(\frac{\lambda_{\mu+1}^\upsilon \sum_{i=\mu+1}^N g_{i,n-j}^2 \lambda_i^{\rho(n)}}{\lambda_\mu^{\bar{q}} G(n-j,\mu) + \sum_{i=\mu+1}^N g_{i,n-j}^2 \lambda_i^{\rho(n)}}\right)^\frac{1}{\upsilon}.
\end{split}
\]
For all $\varepsilon>0$, suppose that
\[
G(n-j,\mu)\le\varepsilon,\quad g_{\mu+1,n-j}^2\ge c_2\varepsilon.
\]
Then,
\[
\left(\alpha_n^\text{DGMR}\right)^{-1} \ge \left(\frac{\lambda_{\mu+1}^\upsilon g_{\mu+1,n-j}^2}{\lambda_\mu^{\bar{q}} \varepsilon + g_{\mu+1,n-j}^2}\right)^\frac{1}{\upsilon} \ge \left(\frac{c_2}{\lambda_\mu^{\bar{q}} + c_2}\right)^\frac{1}{\upsilon} \lambda_{\mu+1}.
\]
Substituting \eqref{eq:c2} into the above deduction it follows that
\[
\left(\alpha_n^\text{DGMR}\right)^{-1} \ge \frac{2}{3}\lambda_{\mu+1},
\] 
which ensures the second condition of Property A.
Thus, the desired conclusion follows by imposing Lemma~\ref{thm:Dai2003:4.1}.
\end{proof}

Notice that the case of $\rho(n) = 0$, $\tau(n) = k$ and $\upsilon = 2$ recovers the AO steplength.
As a consequence of Theorem~\ref{thm:Dai2003:P408}, the convergence result of AOA can be established.
\begin{theorem}
\label{thm:5}
Consider the linear system \eqref{eq:ls} being solved by AOA.
Then the sequence $\{x_n\}$ converges to $x_*$ for any starting point $x_0$.
\end{theorem}
\begin{proof}
The first part of AOA equals exactly the AO steplength which satisfies DGMR framework, thus having the Property A.
The second part can be written as follows
\[
\tilde{\alpha}_n = \theta\left(\frac{g_{\tau(n)}^\intercal g_{\tau(n)}}{g_{\tau(n)}^\intercal A^2 g_{\tau(n)}}\right)^\frac{1}{2}.
\]
As seen in the proof of Theorem~\ref{thm:Dai2003:P408}, we obtain that
\[
\lambda_1 < \alpha_n^{-1} \le \frac{\lambda_N}{\theta}.
\]
Therefore, we can choose $c_1 = \lambda_N/\theta$.
For the second condition, we can keep formula \eqref{eq:c2} for $c_2$, which gives the same result as the deduction for DGMR, and thus AOA has Property A.
Then the desired conclusion follows from Lemma~\ref{thm:Dai2003:4.1}.
\end{proof}

For the convergence of MGA and MGC, we can provide similar theorems.
Notice that the analysis of SDA and SDC can be applied here without difficulty since SD and MG share similar properties as discussed in Section~\ref{sec:2} and Theorem~\ref{thm:4}.
\begin{theorem}
\label{thm:6}
Consider the linear system \eqref{eq:ls} being solved by MGA.
Then the sequence $\{x_n\}$ converges to $x_*$ for any starting point $x_0$.
\end{theorem}
\begin{proof}
This proof follows as before with $m_0=d_2$, $c_1=2\lambda_N$ and $c_2=2$.
For all $j\in\{0,\,\dots,\,\min\{n,m_0\}\}$, let $\alpha_n=\alpha_{n-j+1}^\text{A2}$.
By the fact that
\begin{equation}
\label{eq:6:1}
\frac{1}{2\lambda_N} \le \frac{\min\{\alpha_{n-j}^\text{MG},\,\alpha_{n-j+1}^\text{MG}\}}{2} \le \alpha_n \le \min\{\alpha_{n-j}^\text{MG},\,\alpha_{n-j+1}^\text{MG}\} \le \frac{1}{\lambda_1},
\end{equation}
one can verify that the first property is true.
In addition, since \eqref{eq:6:1} implies that
\[
\alpha_n^{-1} \ge \frac{1}{\min\{\alpha_{n-j}^\text{MG},\,\alpha_{n-j+1}^\text{MG}\}} \ge \frac{1}{\alpha_{n-j}^\text{MG}} = \frac{g_{n-j}^\intercal A^2g_{n-j}}{g_{n-j}^\intercal Ag_{n-j}},
\]
by applying the proof of Theorem~\ref{thm:Dai2003:P408}, it follows that
\[
\alpha_n^{-1} \ge \frac{c_2}{1+c_2}\lambda_{\mu+1}.
\]
Substituting $c_2$ yields the second property.
Thus, the desired conclusion follows from Lemma~\ref{thm:Dai2003:4.1}.
\end{proof}
\begin{theorem}
\label{thm:7}
Consider the linear system \eqref{eq:ls} being solved by MGC.
Then the sequence $\{x_n\}$ converges to $x_*$ for any starting point $x_0$.
\end{theorem}
\begin{proof}
Let $m_0=d_2$.
Similar to the proof of Theorem~\ref{thm:6}, for all $j\in\{0,\,\dots,\,\min\{n,m_0\}\}$, we can write $\alpha_n=\alpha_{n-j+1}^\text{Y2}$.
It is clear that
\begin{equation}
\label{eq:7:1}
\alpha_n \le \min\{\alpha_{n-j}^\text{MG},\,\alpha_{n-j+1}^\text{MG}\} \le \frac{1}{\lambda_1}.
\end{equation}
Given $B$ be an SPD matrix, recall that the Kantorovich inequality~\cite{Saad2003} is of the form
\[
\frac{\left(u^\intercal Bu\right) \left(u^\intercal B^{-1}u\right)}{(u^\intercal u)^2} \le \frac{\left(\lambda_\text{max} + \lambda_\text{min}\right)^2}{4\lambda_\text{max}\lambda_\text{min}},\quad \forall u\neq 0.
\]
It follows that
\[
\frac{g_{n-j+1}^\intercal Ag_{n-j+1}}{g_{n-j}^\intercal Ag_{n-j}} = \frac{g_{n-j}^\intercal Ag_{n-j} \cdot g_{n-j}^\intercal A^3g_{n-j}}{\left(g_{n-j}^\intercal A^2g_{n-j}\right)^2}-1 \le \frac{(\lambda_N-\lambda_1)^2}{4\lambda_N\lambda_1},
\]
from which we can obtain that
\begin{equation}
\label{eq:7:2}
\alpha_n \ge 2\left(\sqrt{(\lambda_N-\lambda_1)^2+\kappa(\lambda_N-\lambda_1)^2}+2\lambda_N\right)^{-1}.
\end{equation}
Since the second member is a constant, combining \eqref{eq:7:1} and \eqref{eq:7:2} yields the first property.
Finally, comparing \eqref{eq:7:1} with \eqref{eq:6:1} implies that the second result can be obtained in the same manner as that follows from the proof of Theorem~\ref{thm:6}.
Thus, we arrive at the desired conclusion.
\end{proof}

\section{Numerical experiments}
\label{sec:5}

In this section, we provide numerical experiments for different gradient methods by two types of problems.
The first one is generated randomly by MATLAB and the second one is a two-point boundary value problem.
In both examples, the right-hand side $b$ in system~\eqref{eq:ls} is computed by $b=Ax_*$ where $x_*$ is a random vector such that $x_*\in(-10,\,10)$.
The tests are started from zero vectors and the stopping criterion is fixed with $\norm{g_n} < 10^{-6}\norm{g_0}$.
All experiments are performed using MATLAB R2018b on a machine with Double Intel Core i7 2.8 GHz CPU.

In the first example, we consider the random problem generated by the MATLAB built-in function \texttt{sprandsym}, which has appeared in~\cite{DeAsmundis2013}.
We would like to know the impact of parameters on the convergence behavior of alternate gradient methods.
The plots in Figs.~\ref{fig:1} and~\ref{fig:2} show some examples where AOA, SDC and MGC are used for solving random problems.
\begin{figure}[!t]
\centering
\begin{subfigure}{.50\textwidth}
  \centering
  \includegraphics[width=1.\linewidth]{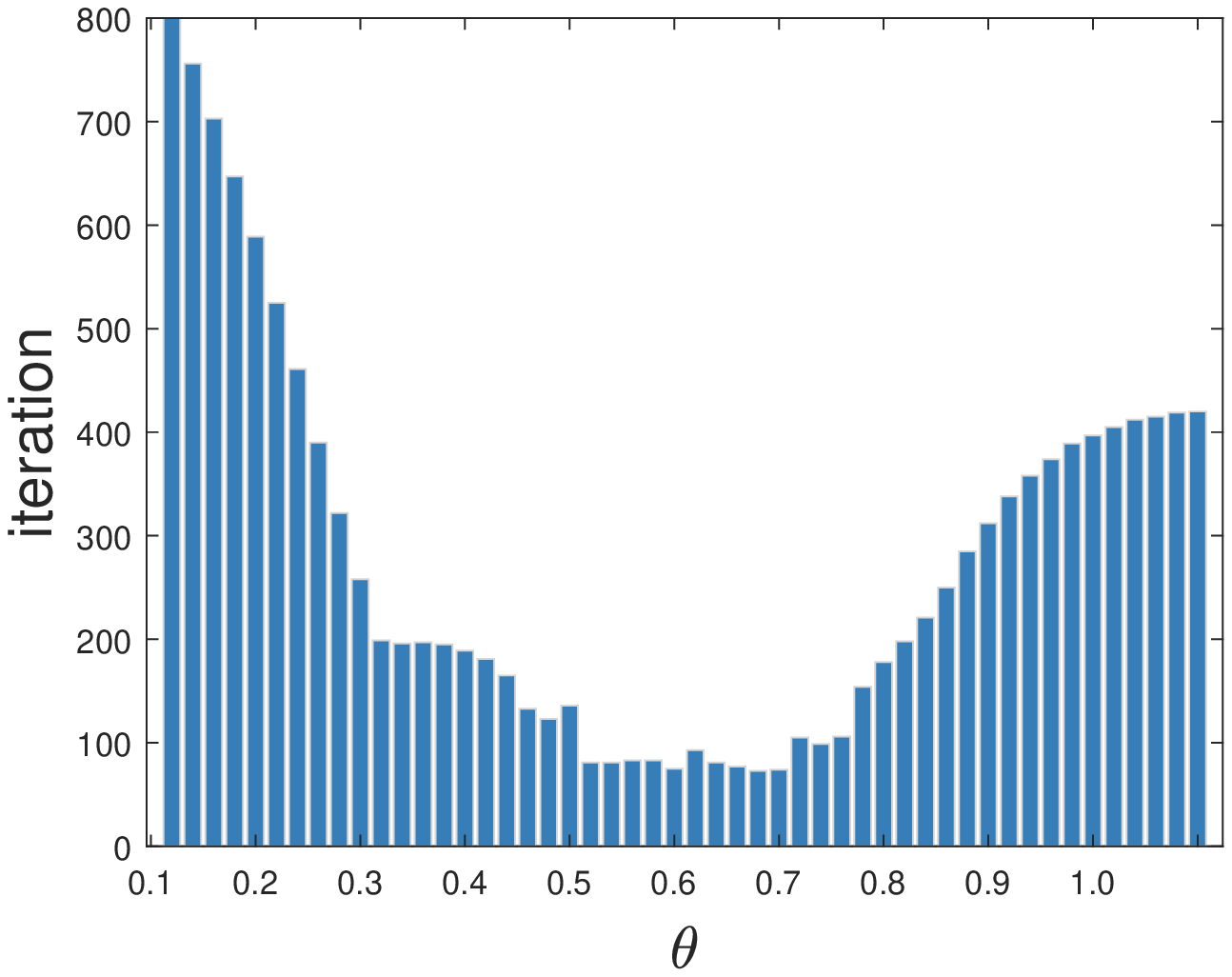}
\end{subfigure}\begin{subfigure}{.50\textwidth}
  \centering
  \includegraphics[width=1.\linewidth]{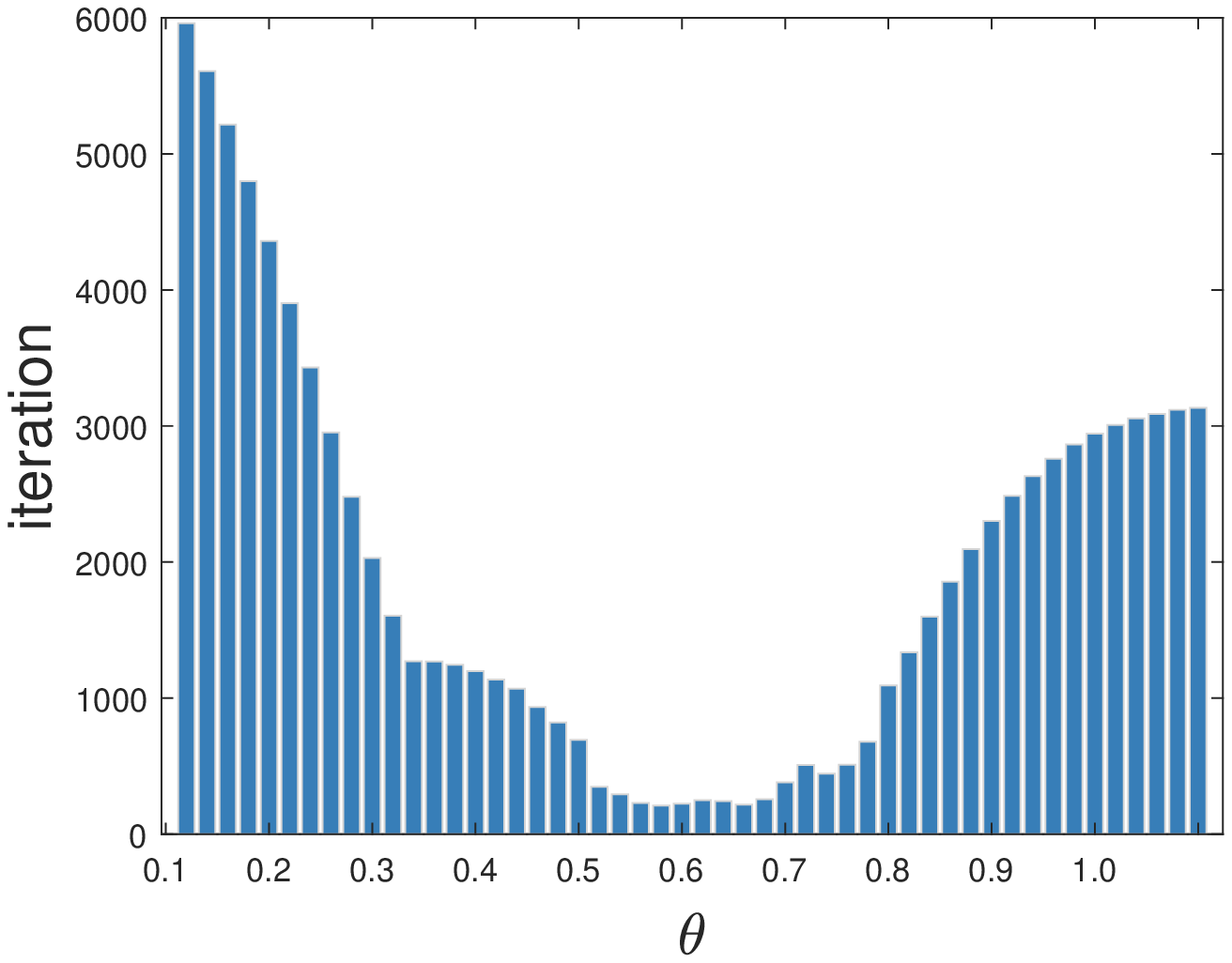}
\end{subfigure}
\caption{Comparison of different $\theta$ in AOA where $d_1=4$ and $d_2=4$. We generate random problems with~$N=100$: $\kappa=10^2$ (left), $\kappa=10^3$ (right).}
\label{fig:1}
\end{figure}
\begin{figure}[!t]
\centering
\begin{subfigure}{.33\textwidth}
  \centering
  \includegraphics[width=1.\linewidth]{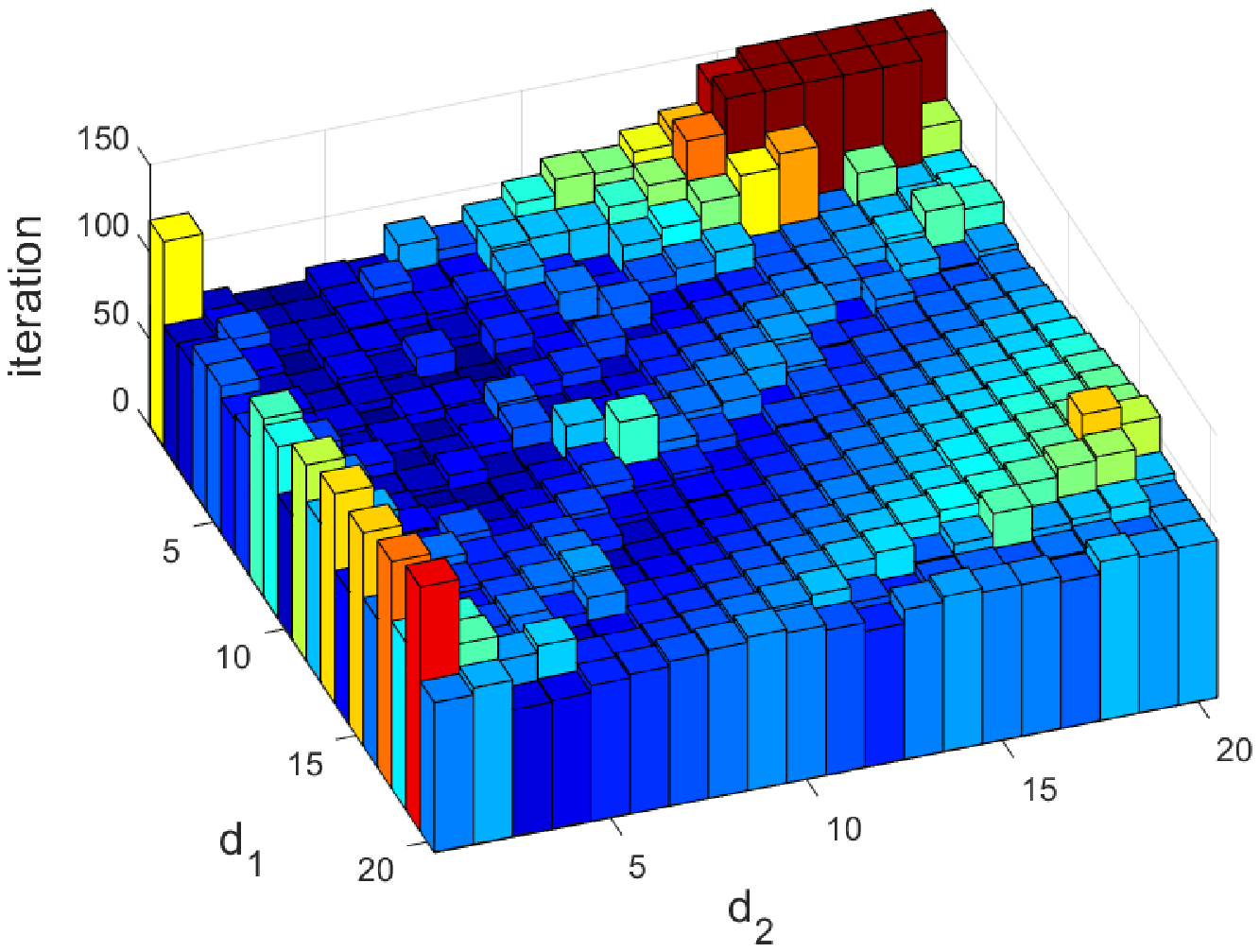}
\end{subfigure}\begin{subfigure}{.33\textwidth}
  \centering
  \includegraphics[width=1.\linewidth]{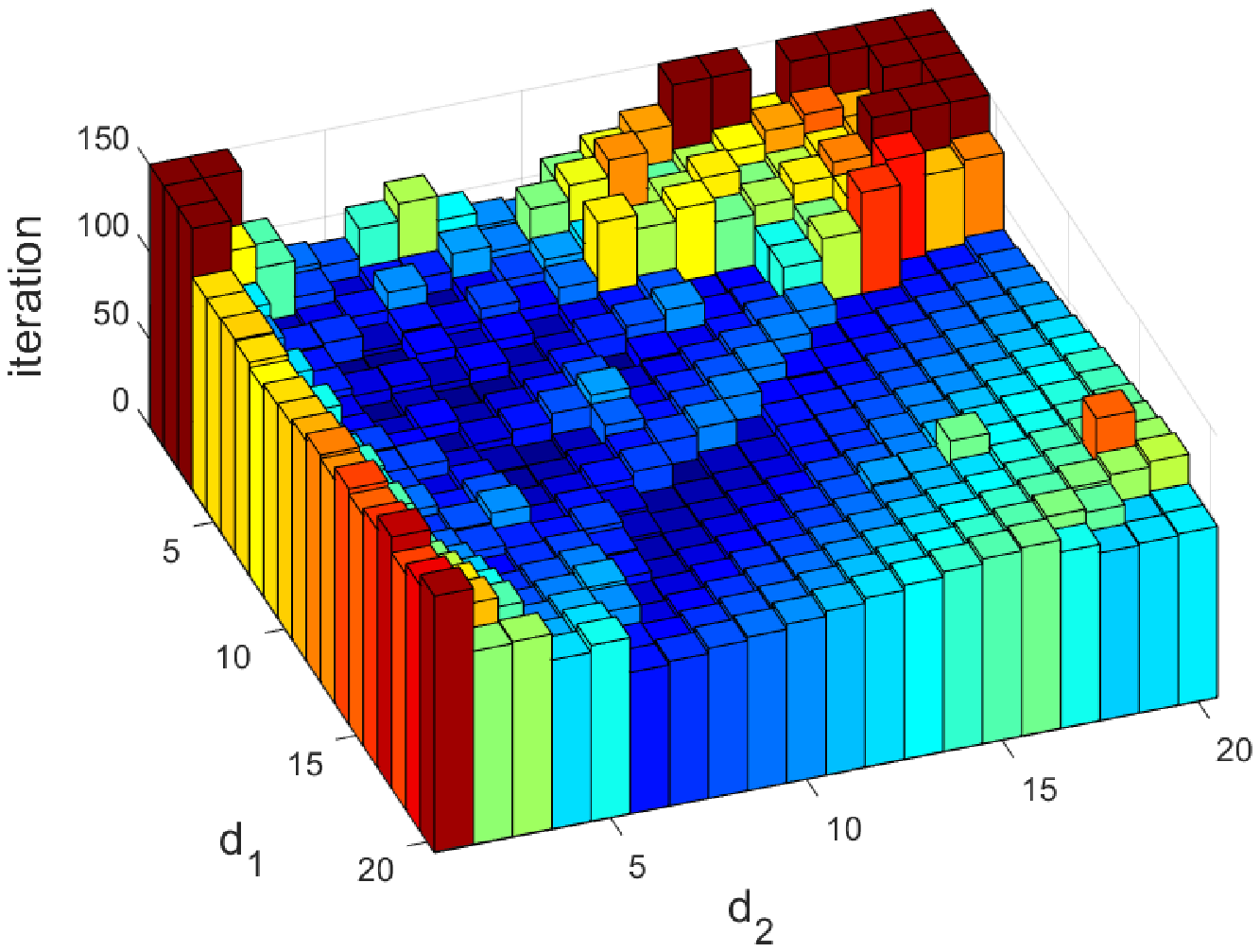}
\end{subfigure}\begin{subfigure}{.33\textwidth}
  \centering
  \includegraphics[width=1.\linewidth]{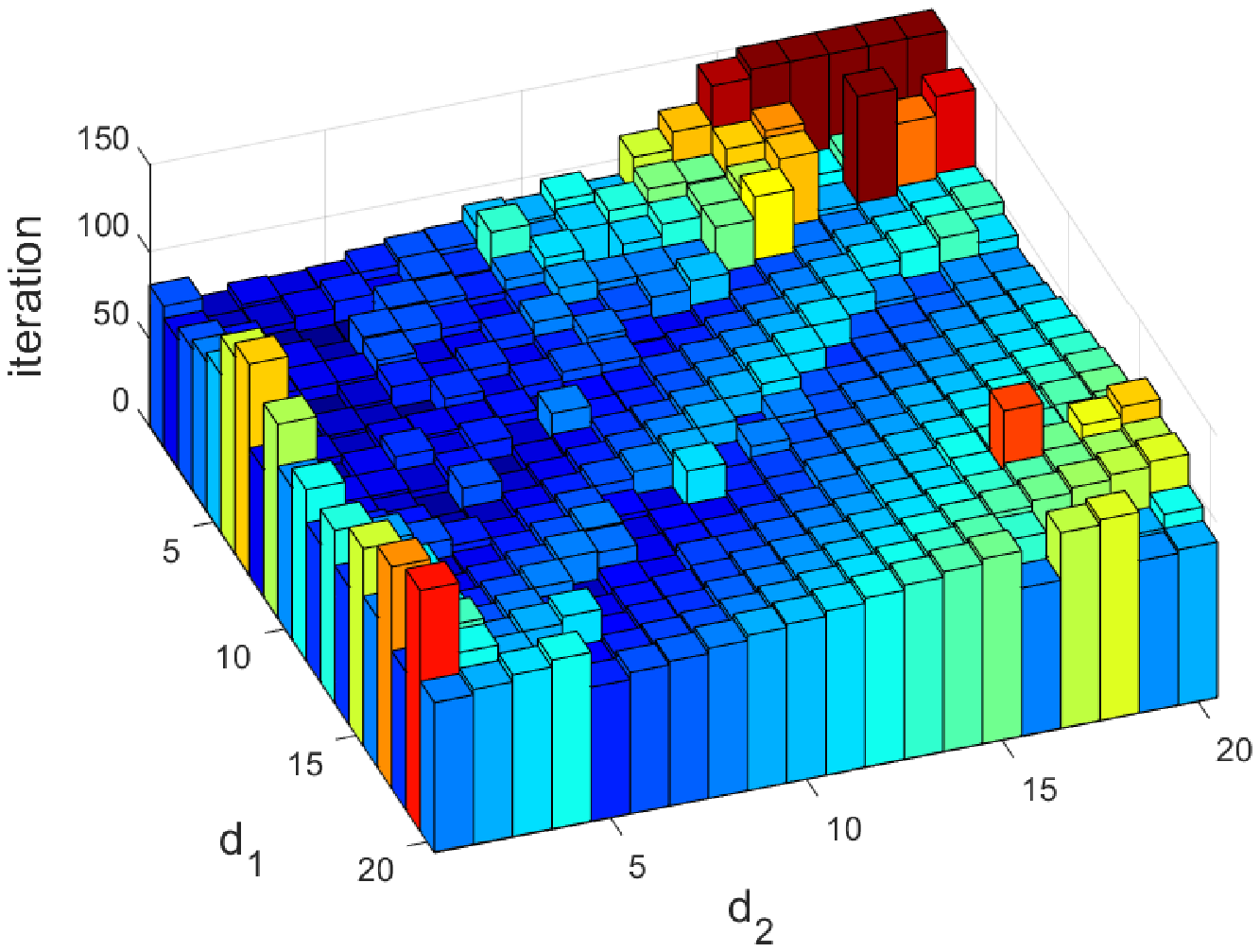}
\end{subfigure}
\begin{subfigure}{.33\textwidth}
  \centering
  \includegraphics[width=1.\linewidth]{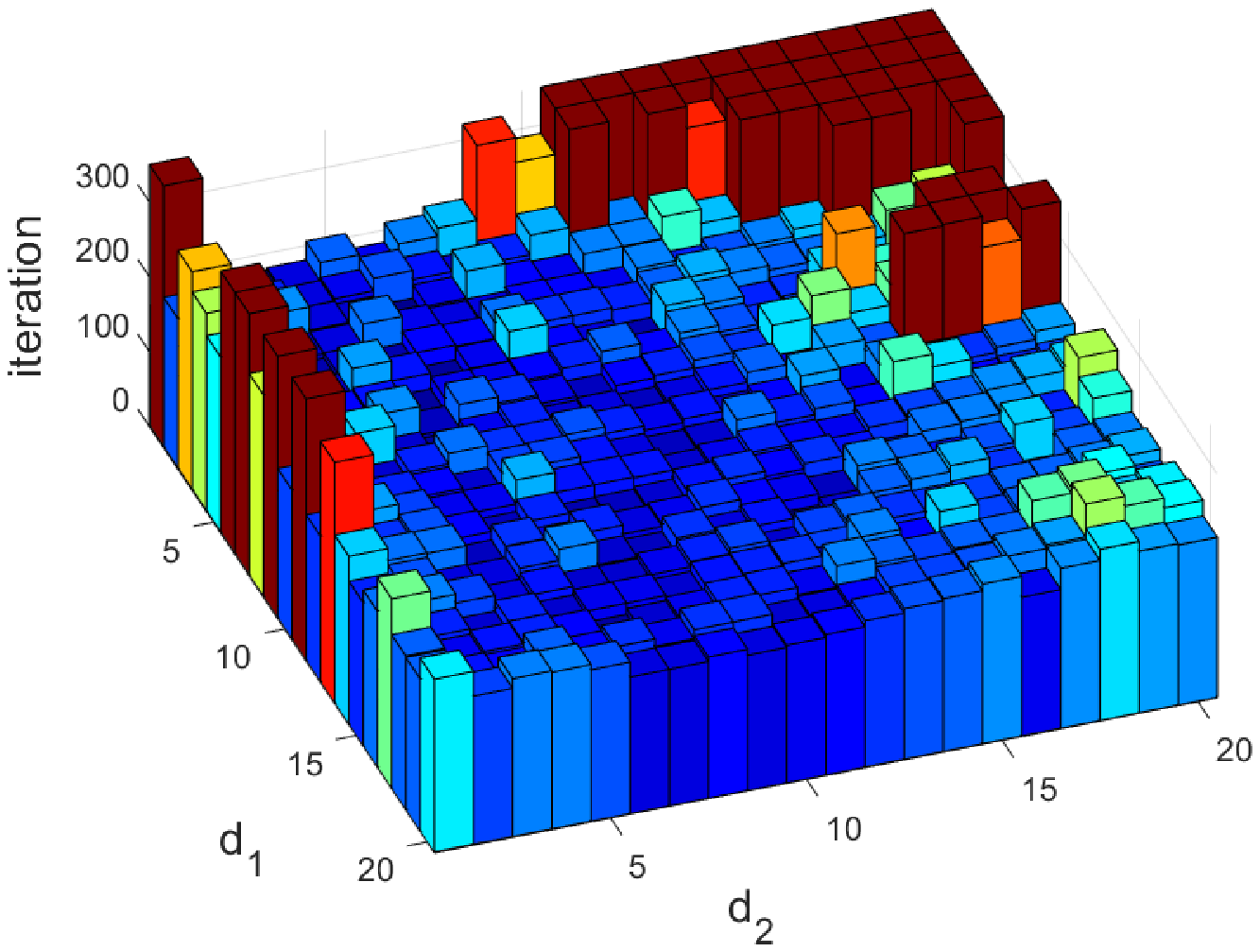}
\end{subfigure}\begin{subfigure}{.33\textwidth}
  \centering
  \includegraphics[width=1.\linewidth]{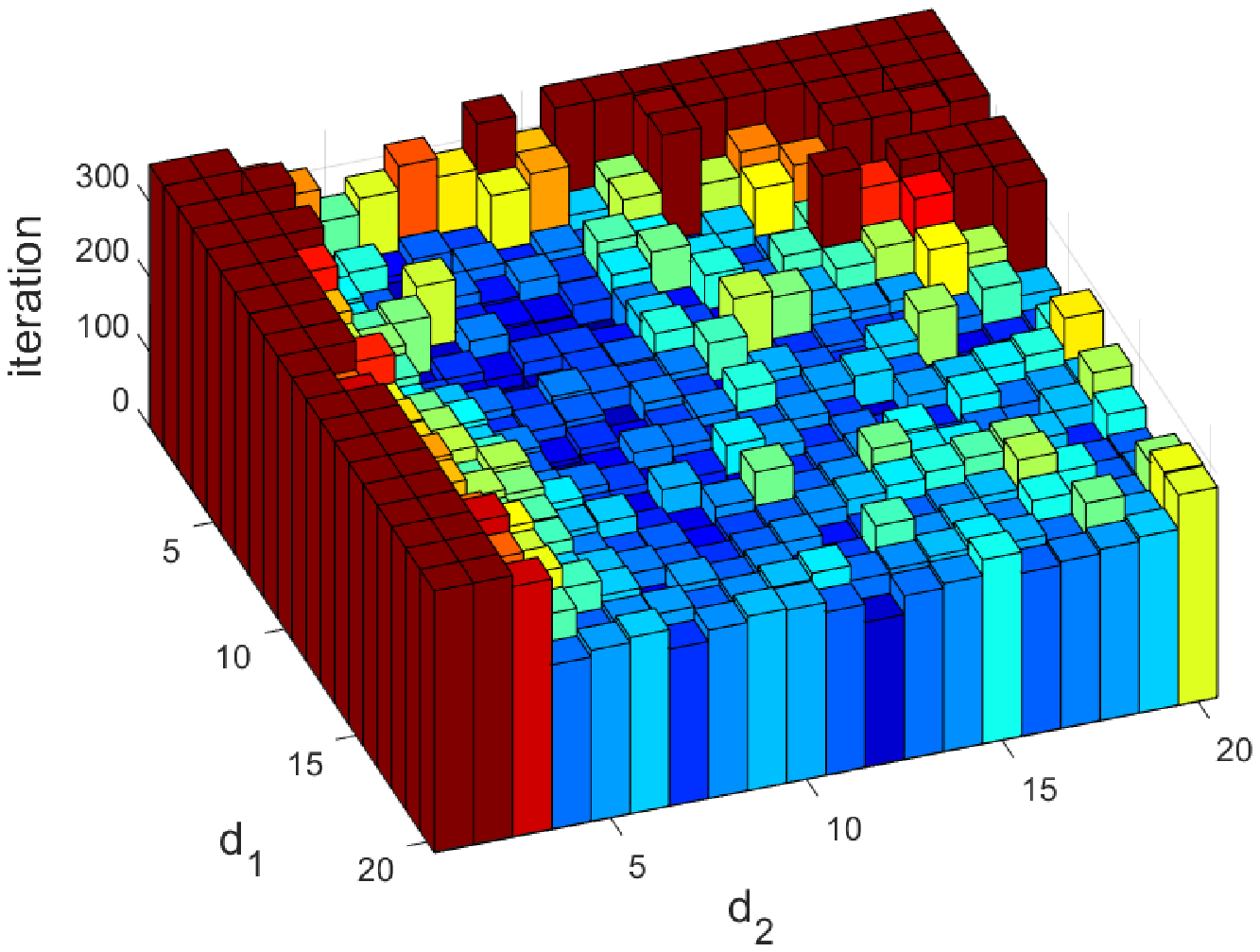}
\end{subfigure}\begin{subfigure}{.33\textwidth}
  \centering
  \includegraphics[width=1.\linewidth]{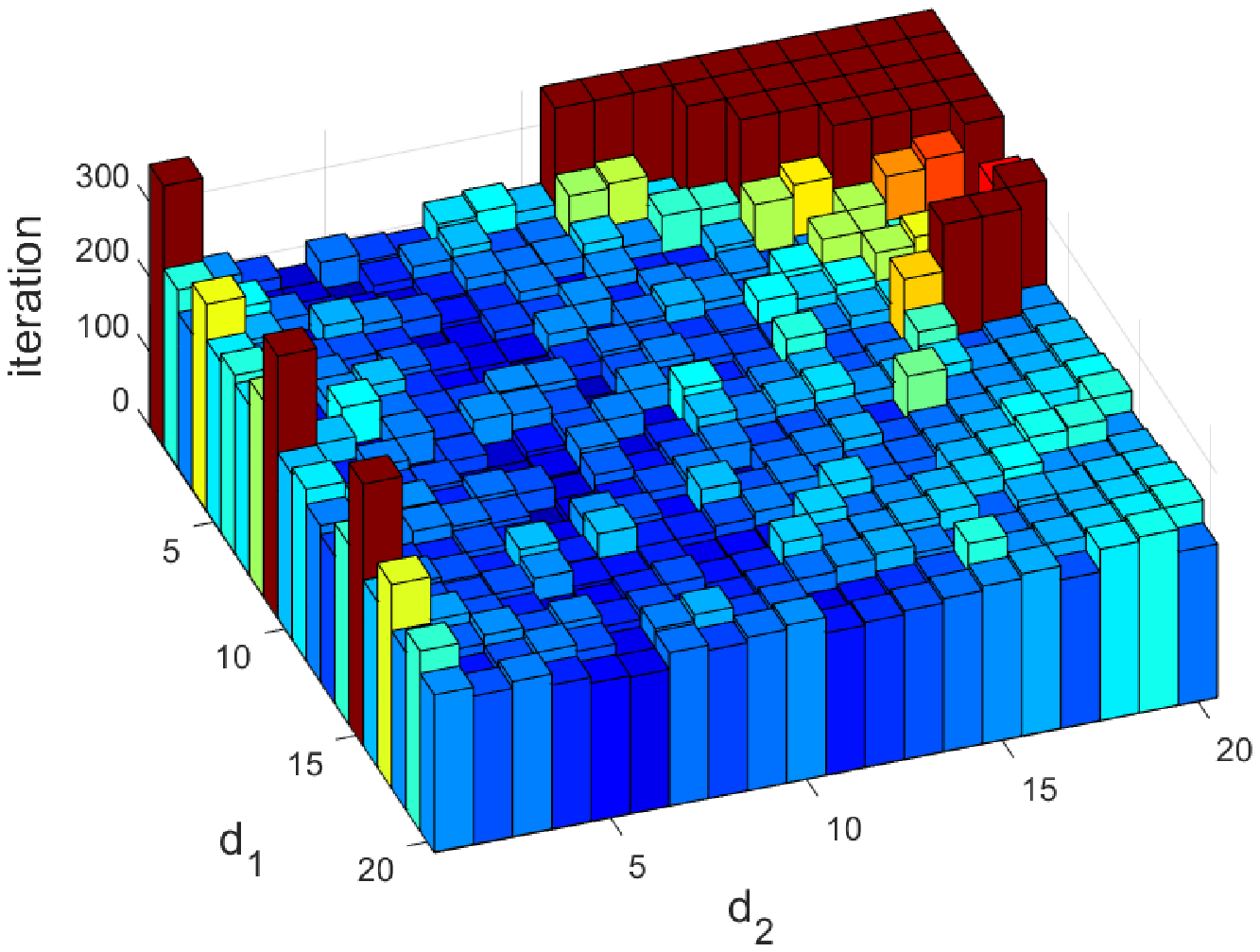}
\end{subfigure}
\caption{Comparison of SDC (left), AOA (center) and MGC (right) through random problems with~$N=100$: $\kappa=10^2$ (top), $\kappa=10^3$ (bottom).}
\label{fig:2}
\end{figure}
Figs.~\ref{fig:1} illustrates the impact of parameter $\theta$ on AOA iterations.
We can see that $\theta\in[0.5,\,0.7]$ leads to the most efficient algorithm.
In Figs.~\ref{fig:2}, we notice that the blue areas illustrate the situation where the choice of parameters leads to fast convergence, while the red ones show the opposite results.
It is convenient to propose an adaptive way to select parameters according to the matrix dimension and the distribution of eigenvalues, but the spectral property is generally unknown to us and obtaining the distribution of eigenvalues is as difficult as solving a linear system.

In the following experiments, we choose $\theta=0.5$ for AOA and $d_1=4$ and $d_2=4$ for all methods since according to Figs.~\ref{fig:1} and~\ref{fig:2} they often produce good results.
Fig.~\ref{fig:3} shows the convergence behaviors of several typical gradient methods.
\begin{figure}[!t]
\centering
\begin{subfigure}{.5\textwidth}
  \centering
  \includegraphics[width=1.\linewidth]{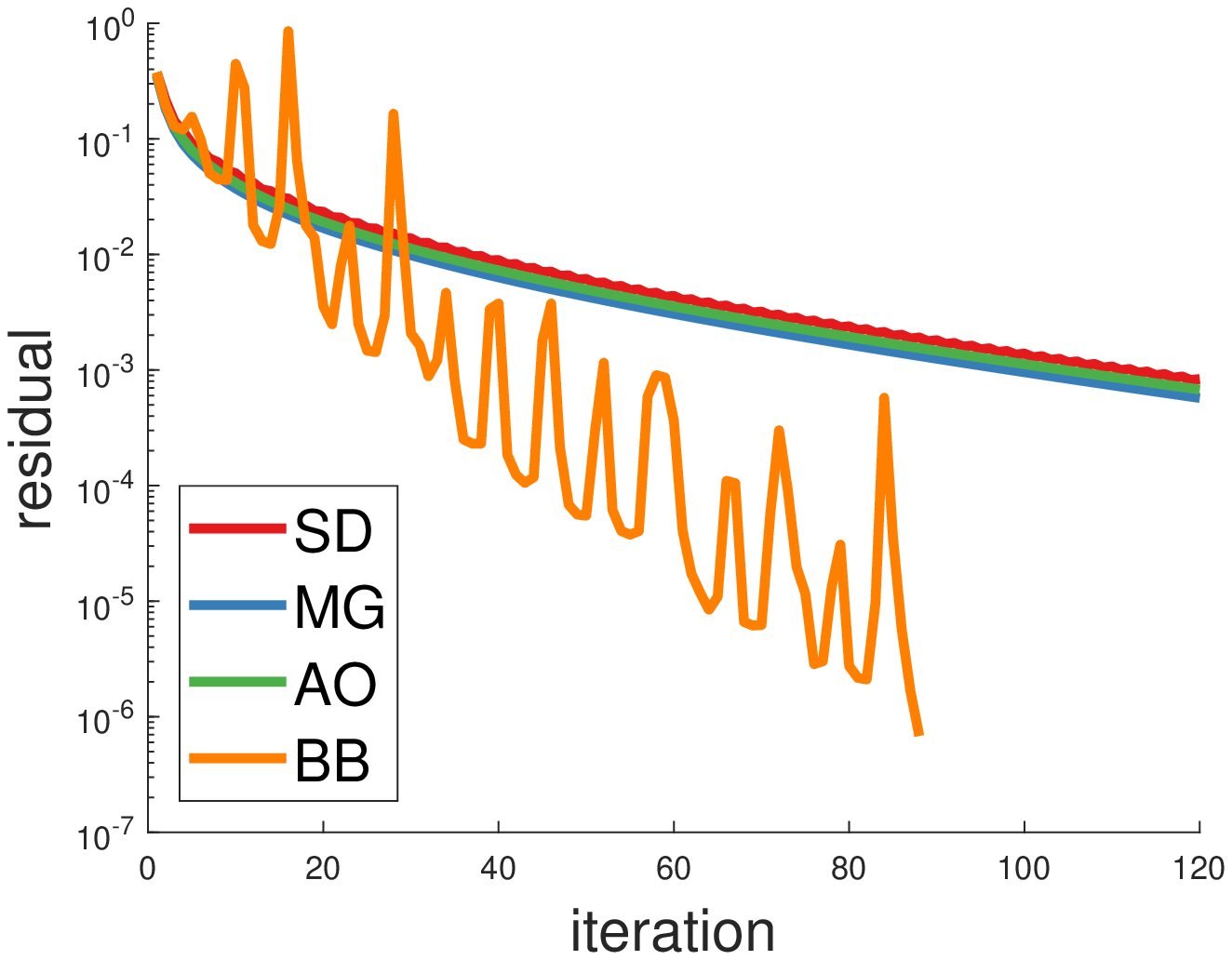}
\end{subfigure}\begin{subfigure}{.5\textwidth}
  \centering
  \includegraphics[width=1.\linewidth]{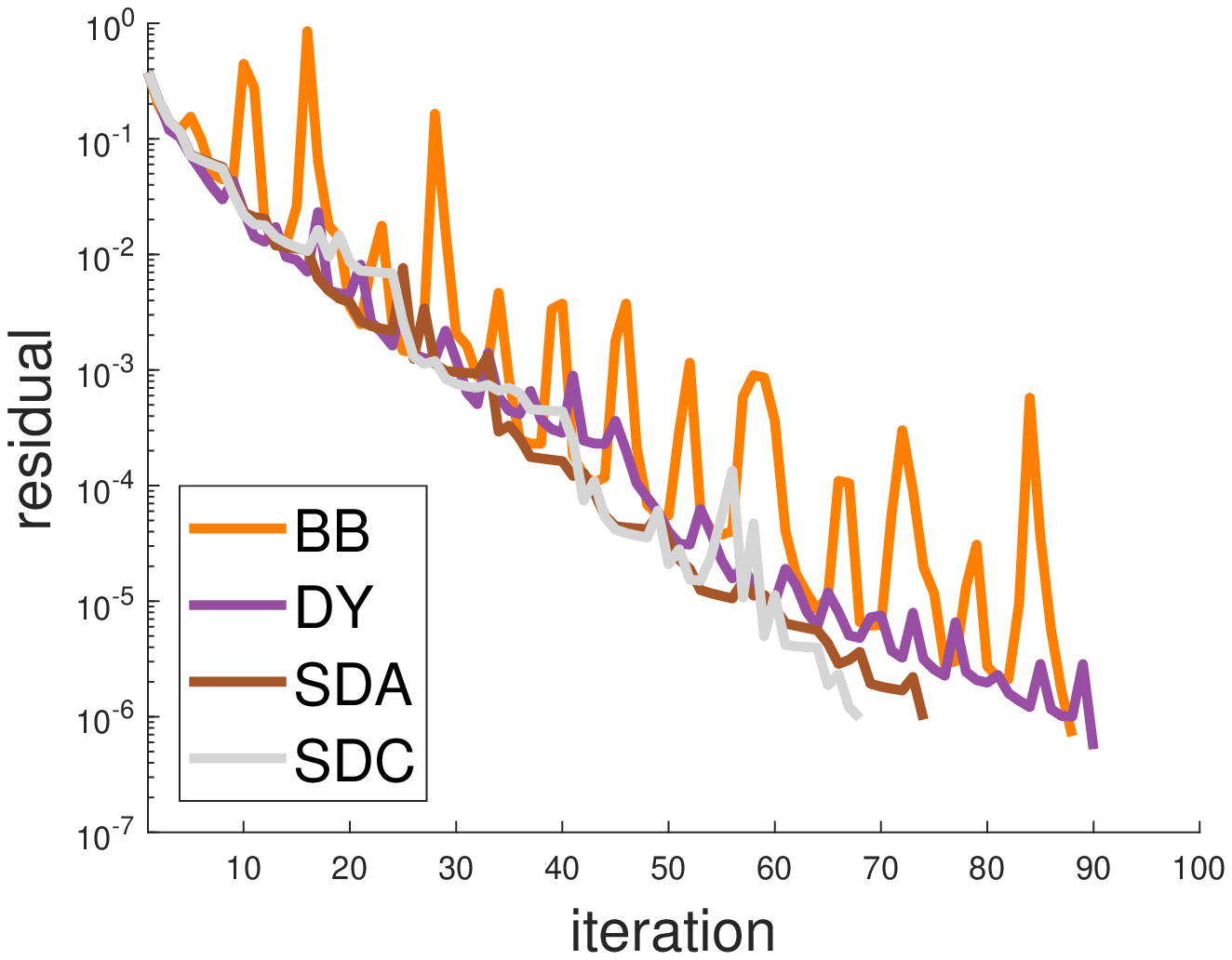}
\end{subfigure}
\begin{subfigure}{.5\textwidth}
  \centering
  \includegraphics[width=1.\linewidth]{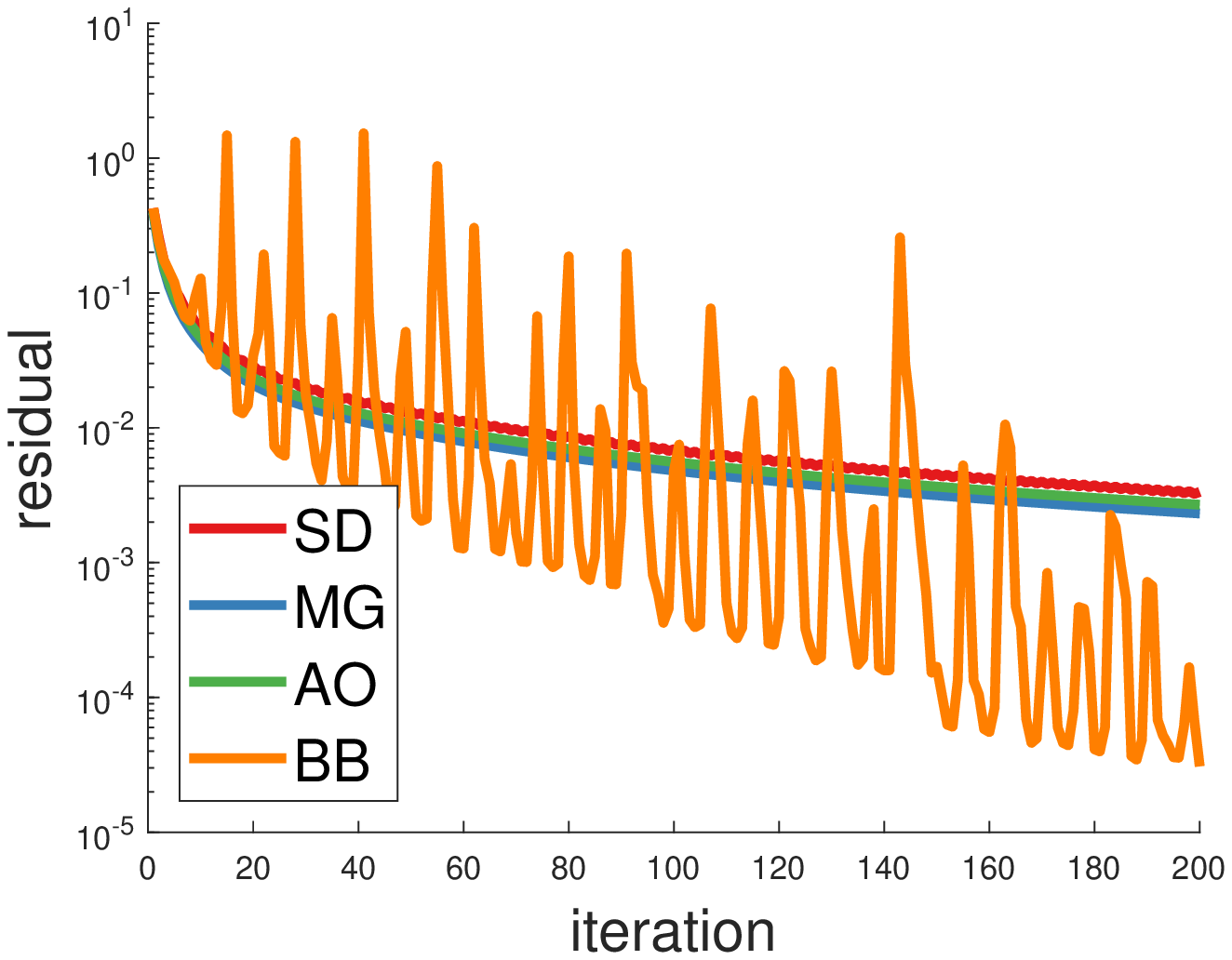}
\end{subfigure}\begin{subfigure}{.5\textwidth}
  \centering
  \includegraphics[width=1.\linewidth]{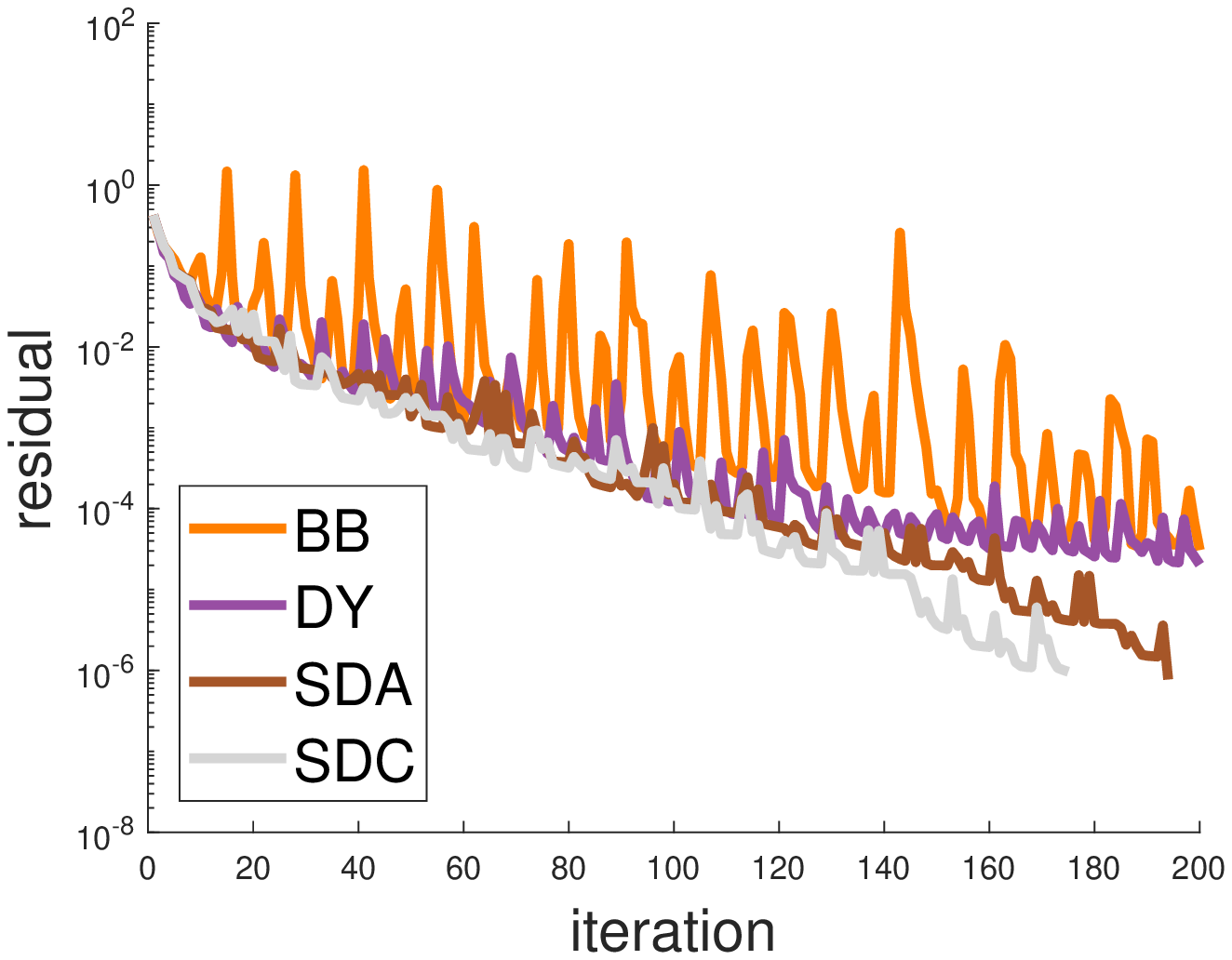}
\end{subfigure}
\caption{Comparison of different gradient methods through random problems: $N=100,\,\kappa=100$ (top), $N=100,\,\kappa=1000$ (bottom).}
\label{fig:3}
\end{figure}
Our tests reveal that the basic methods such as SD, MG and AO are far less efficient than others.
The traditional gradient steps are unrealistic to be used in practice, especially for ill-conditioned problems.
In addition, the convergence results of SDA and SDC are not slower than BB and DY in most cases.
Notice that DY has nonmonotone curve in the residual figure, though it would show monotone behavior when drawing the values of function $f$.

In Table~\ref{tab:1}, we provide the number of iterations required by SDA and SDC as well as the new methods with $\kappa=10^2,\,10^3,\,10^4,\,10^5$ and $N=200,\,400,\,600,\,800,\,1000$.
In all cases, we list only the final average results in the table for which $10$ repeated experiments were conducted to circumvent the extreme conditions.
One finds that SDC and MGC give better results than other three methods.
On the other hand, SDA deteriorates when $\kappa$ becomes larger, and the comparison between AOA and MGA could not lead to a commun conclusion.
This observation is contrary to our expectations, as we speculated that AOA would always have bad performance, due to its twofold asymptotically zigzag behavior, as mentioned in Section~\ref{sec:3}.
Further tests have shown that AOA is more sensitive to the choice of parameters than MGA and MGC.
The problem size seems to be a less critical issue in view of the test results.
\begin{table}[!t]
\caption{The following results are obtained for the problems generated randomly by the MATLAB built-in function \texttt{sprandsym}. In the table we illustrate the average number of iterations among $10$ tests with $d_1=4$ and $d_2=4$ for all methods.}
\small
\begin{tabular}{@{}ccccccc@{}}
\hline
Conditioning & Size & SDA & SDC & AOA & MGA & MGC \\[1pt]
\hline
$\kappa=10^2$ & $N=200$ & 68 & 67 & 80 & 73 & 70 \\[1pt]
& $N=400$ & 70 & 69 & 80 & 73 & 66 \\[1pt]
& $N=600$ & 73 & 72 & 83 & 73 & 73 \\[1pt]
& $N=800$ & 71 & 74 & 81 & 73 & 74 \\[1pt]
& $N=1000$ & 70 & 76 & 80 & 74 & 75 \\ [2pt]
$\kappa=10^3$ & $N=200$ & 199 & 177 & 197 & 209 & 187 \\[1pt]
& $N=400$ & 201 & 187 & 222 & 216 & 190 \\[1pt]
& $N=600$ & 199 & 195 & 226 & 205 & 181 \\[1pt]
& $N=800$ & 191 & 185 & 232 & 207 & 181 \\[1pt]
& $N=1000$ & 194 & 182 & 227 & 209 & 190 \\ [2pt]
$\kappa=10^4$ & $N=200$ & 614 & 479 & 571 & 536 & 507 \\[1pt]
& $N=400$ & 648 & 506 & 525 & 525 & 501 \\[1pt]
& $N=600$ & 602 & 497 & 560 & 540 & 490 \\[1pt]
& $N=800$ & 626 & 484 & 534 & 536 & 509 \\[1pt]
& $N=1000$ & 619 & 475 & 547 & 515 & 488 \\ [2pt]
$\kappa=10^5$ & $N=200$ & 1300 & 1118 & 1246 & 1225 & 1153 \\[1pt]
& $N=400$ & 1318 & 1176 & 1393 & 1299 & 1126 \\[1pt]
& $N=600$ & 1374 & 1228 & 1255 & 1253 & 1231 \\[1pt]
& $N=800$ & 1390 & 1190 & 1452 & 1269 & 1169 \\[1pt]
& $N=1000$ & 1381 & 1273 & 1490 & 1321 & 1251 \\[2pt]
\hline
\end{tabular}
\label{tab:1}
\end{table}

To show the correctness of our analysis, particularly, the comparisons between the aligned methods and the basic gradient methods are illustrated in Fig.~\ref{fig:4}.
\begin{figure}[!t]
\centering
\begin{subfigure}{.25\textwidth}
  \centering
  \includegraphics[width=1.\linewidth]{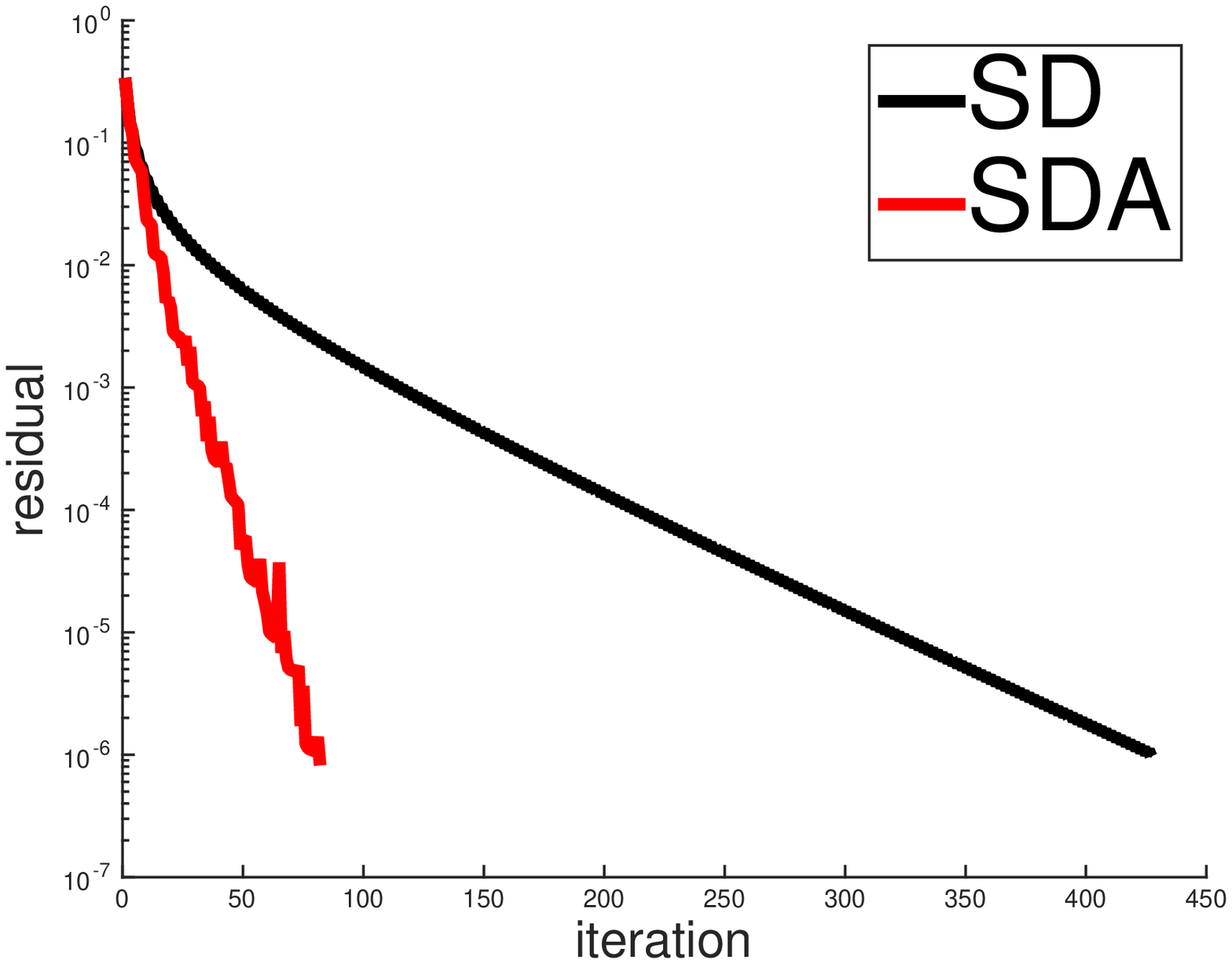}
\end{subfigure}\begin{subfigure}{.25\textwidth}
  \centering
  \includegraphics[width=1.\linewidth]{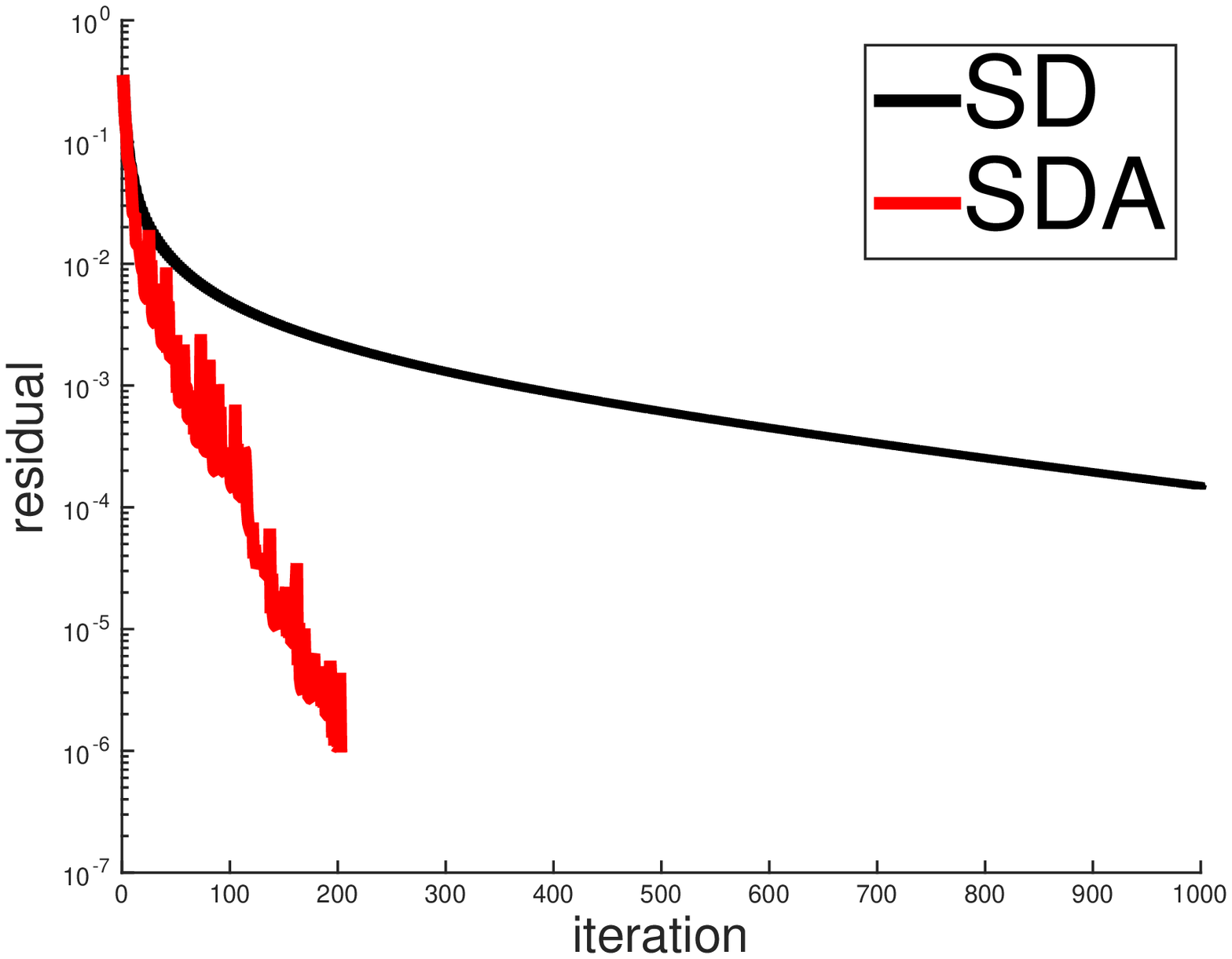}
\end{subfigure}\begin{subfigure}{.25\textwidth}
  \centering
  \includegraphics[width=1.\linewidth]{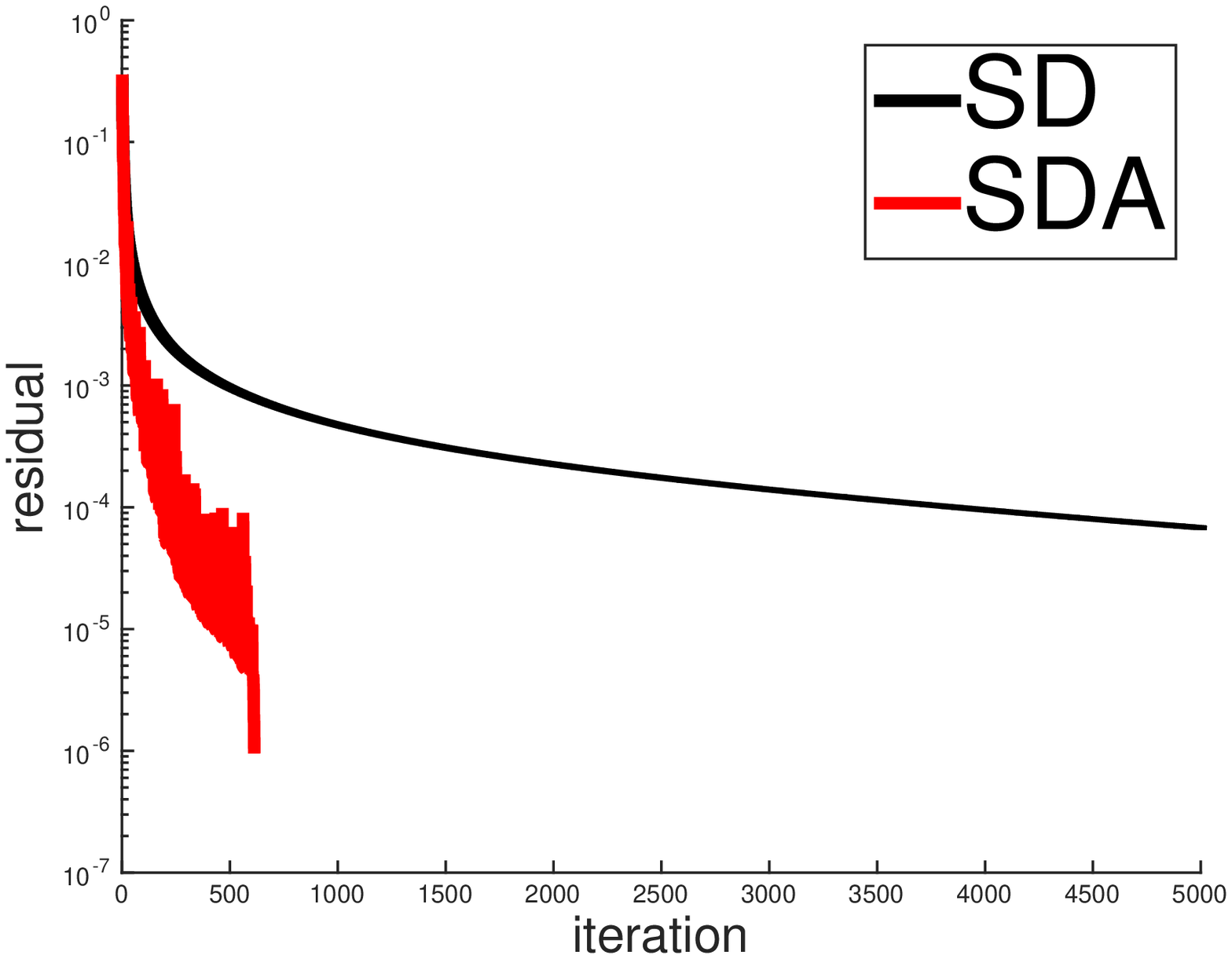}
\end{subfigure}\begin{subfigure}{.25\textwidth}
  \centering
  \includegraphics[width=1.\linewidth]{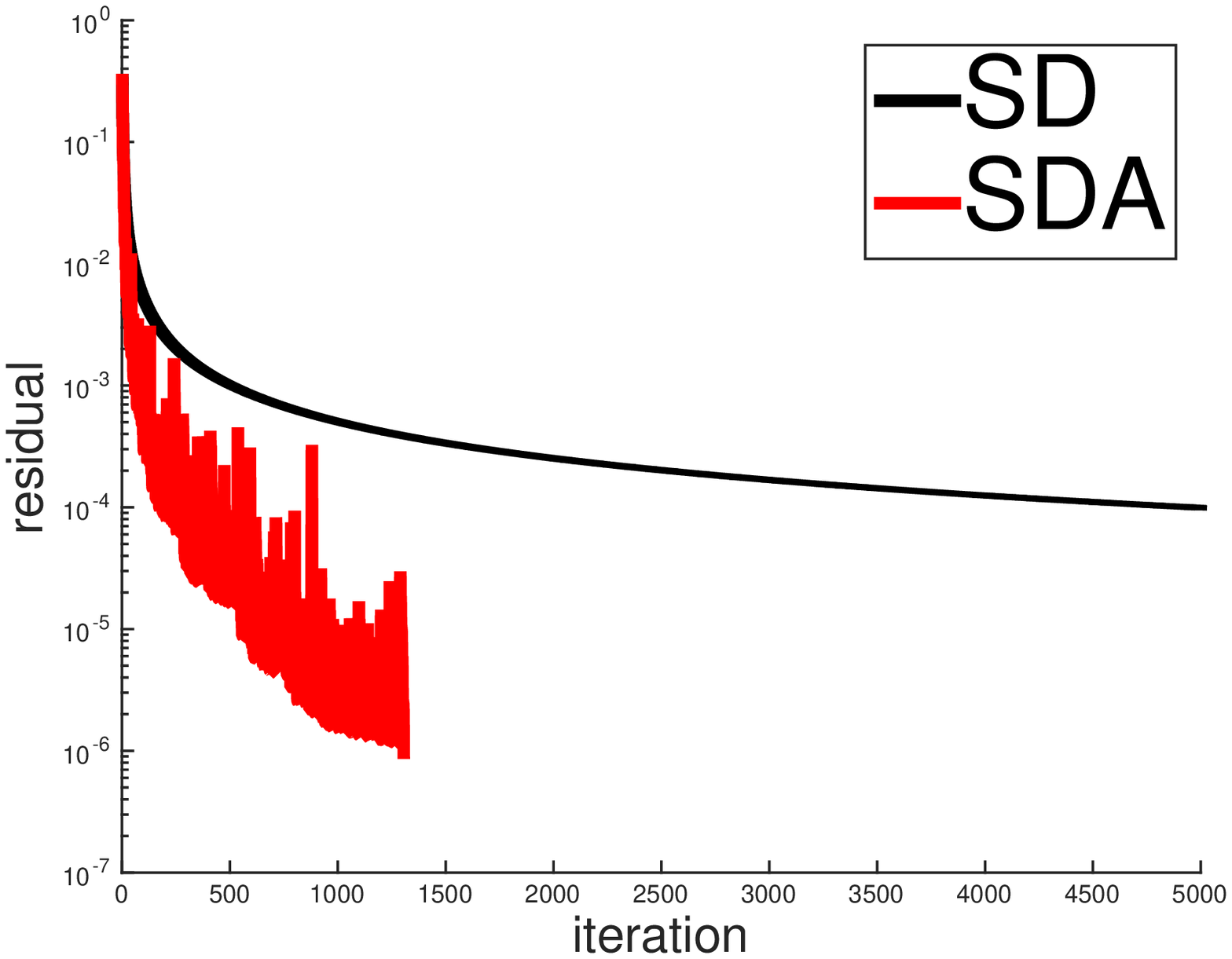}
\end{subfigure}
\begin{subfigure}{.25\textwidth}
  \centering
  \includegraphics[width=1.\linewidth]{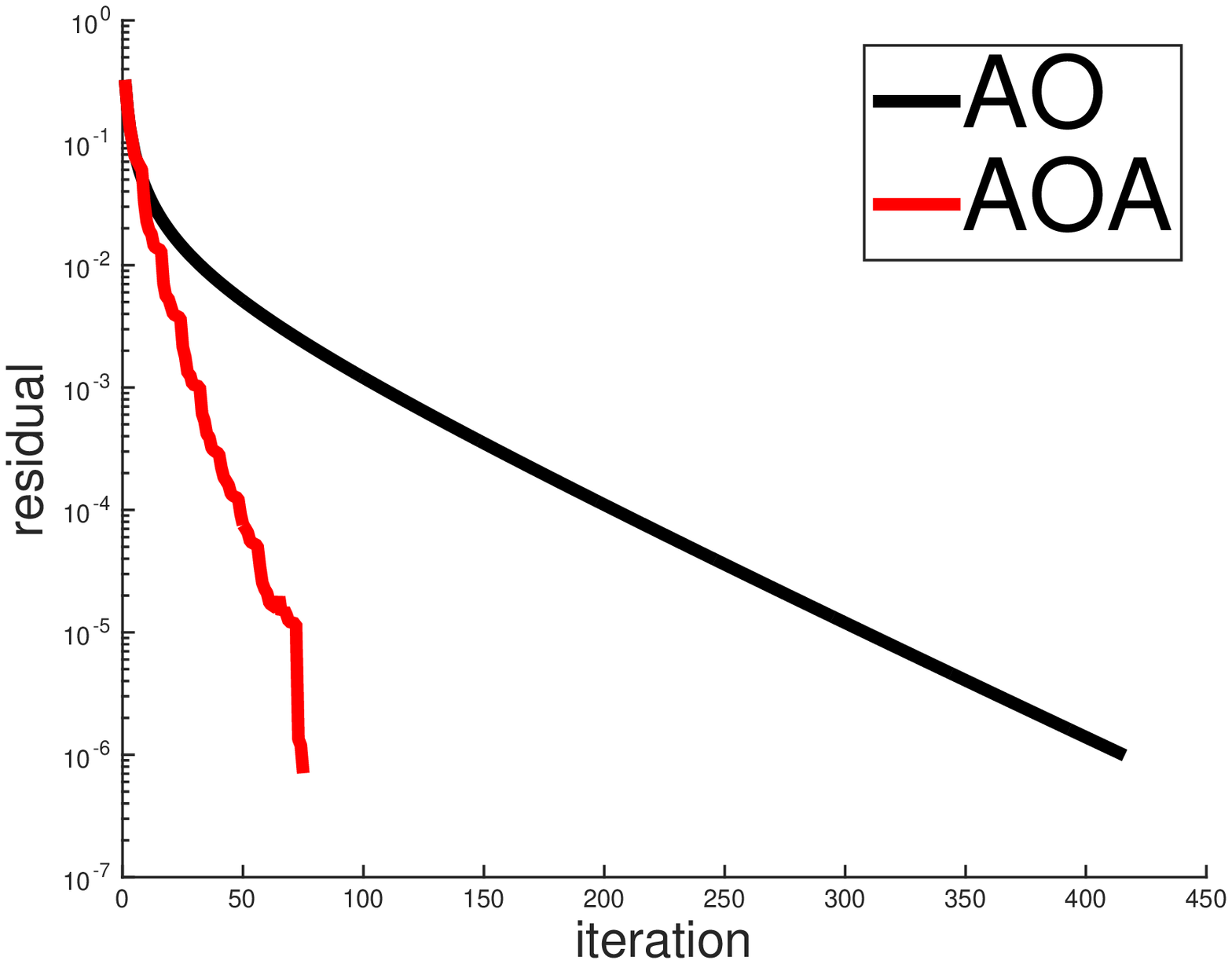}
\end{subfigure}\begin{subfigure}{.25\textwidth}
  \centering
  \includegraphics[width=1.\linewidth]{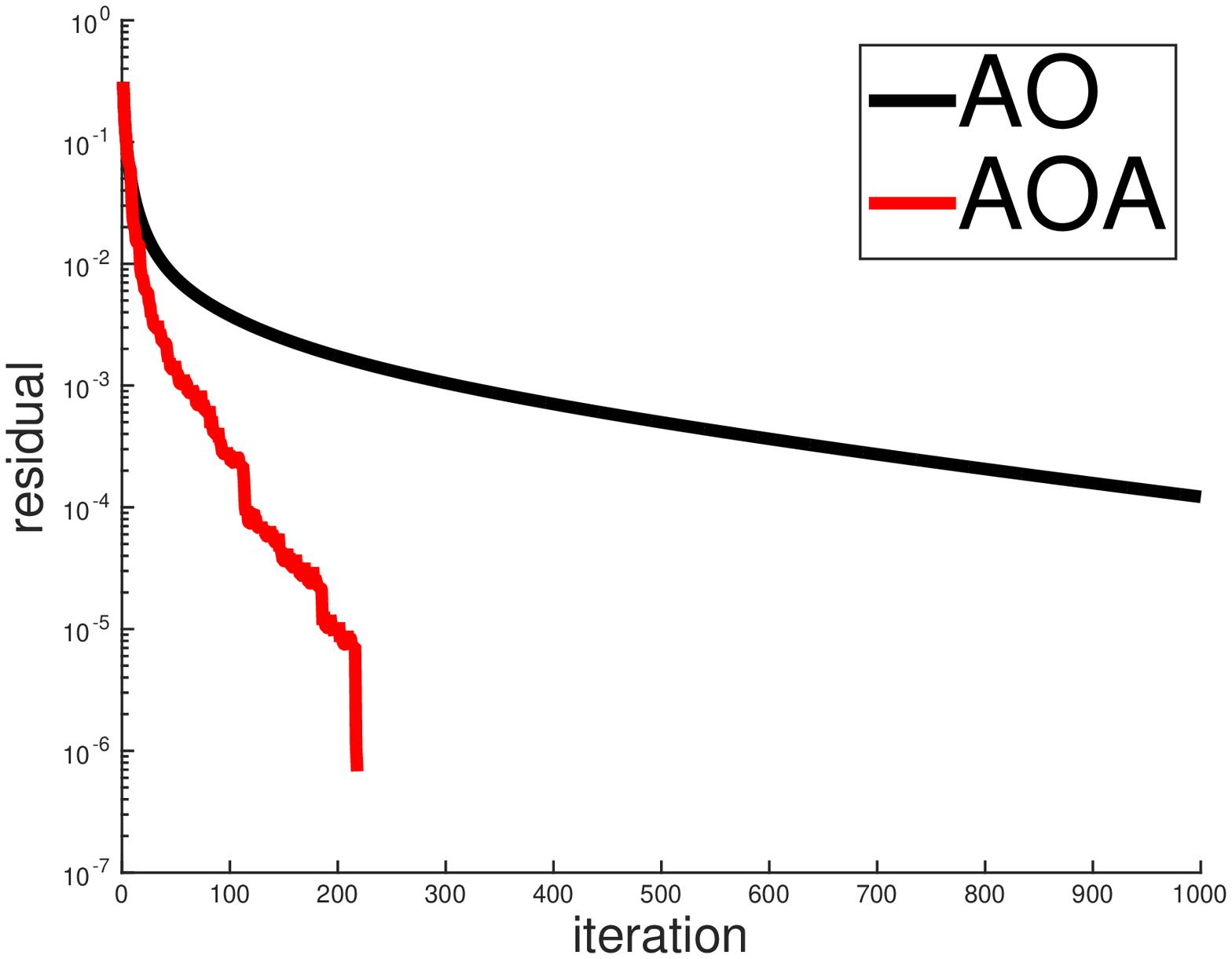}
\end{subfigure}\begin{subfigure}{.25\textwidth}
  \centering
  \includegraphics[width=1.\linewidth]{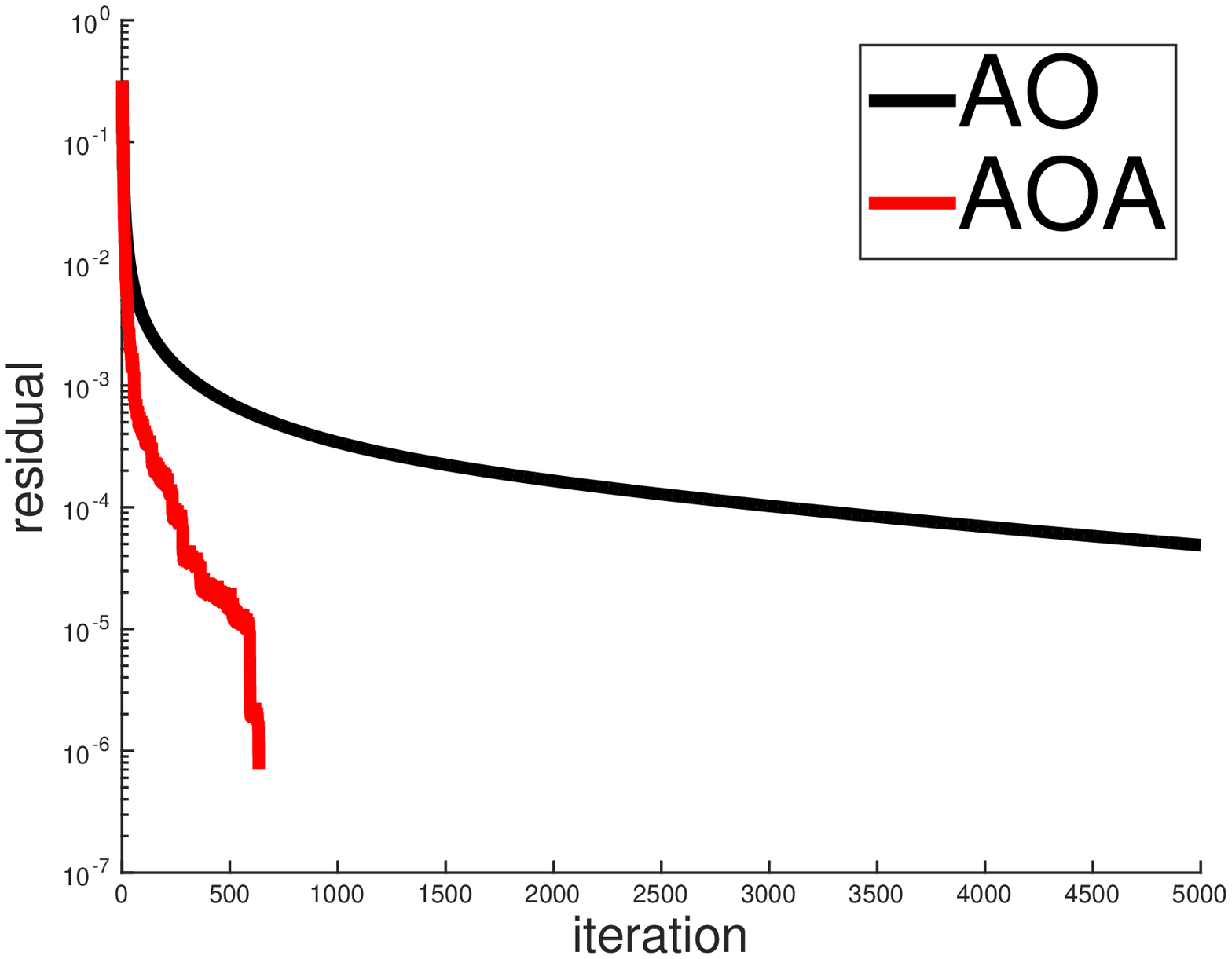}
\end{subfigure}\begin{subfigure}{.25\textwidth}
  \centering
  \includegraphics[width=1.\linewidth]{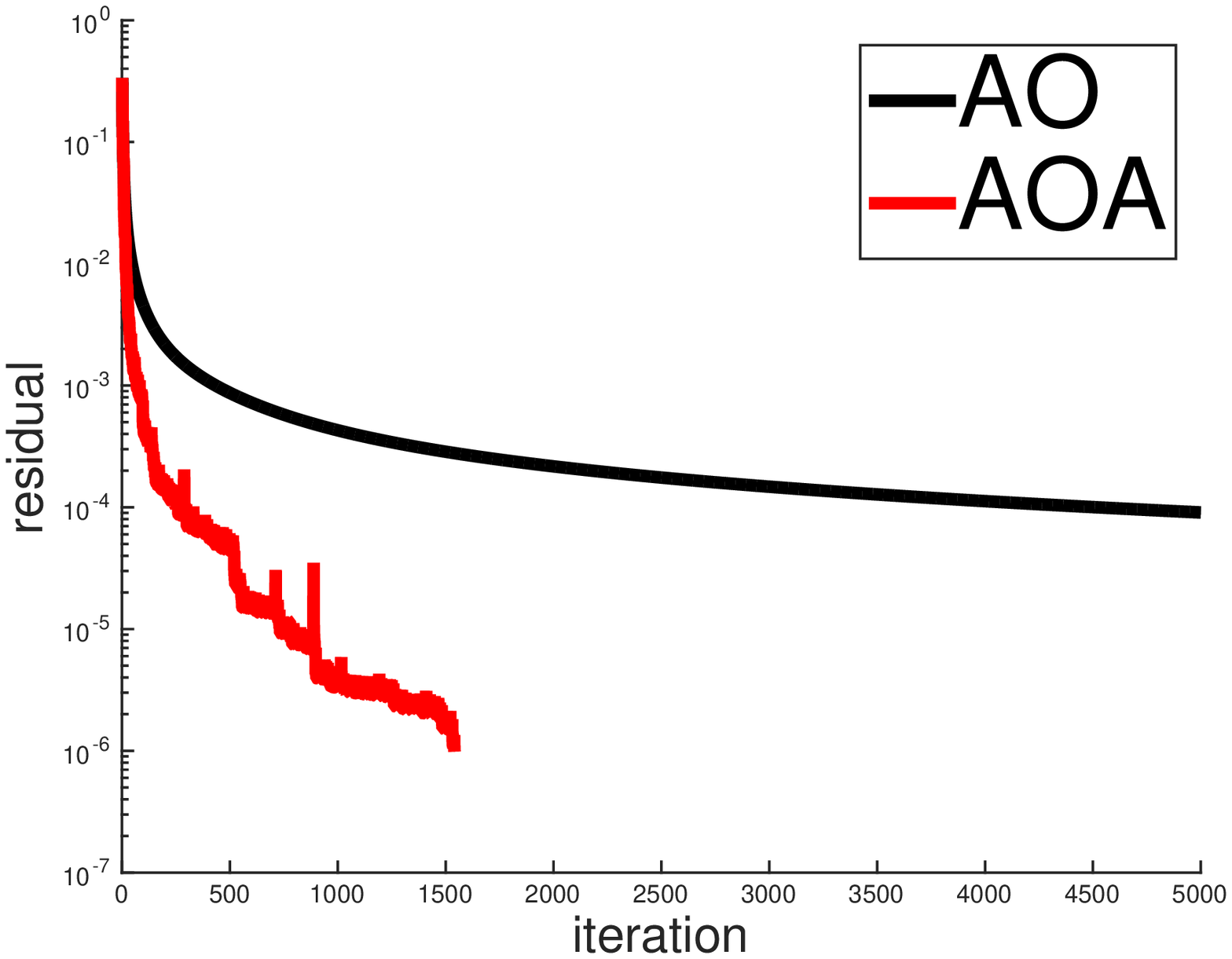}
\end{subfigure}
\begin{subfigure}{.25\textwidth}
  \centering
  \includegraphics[width=1.\linewidth]{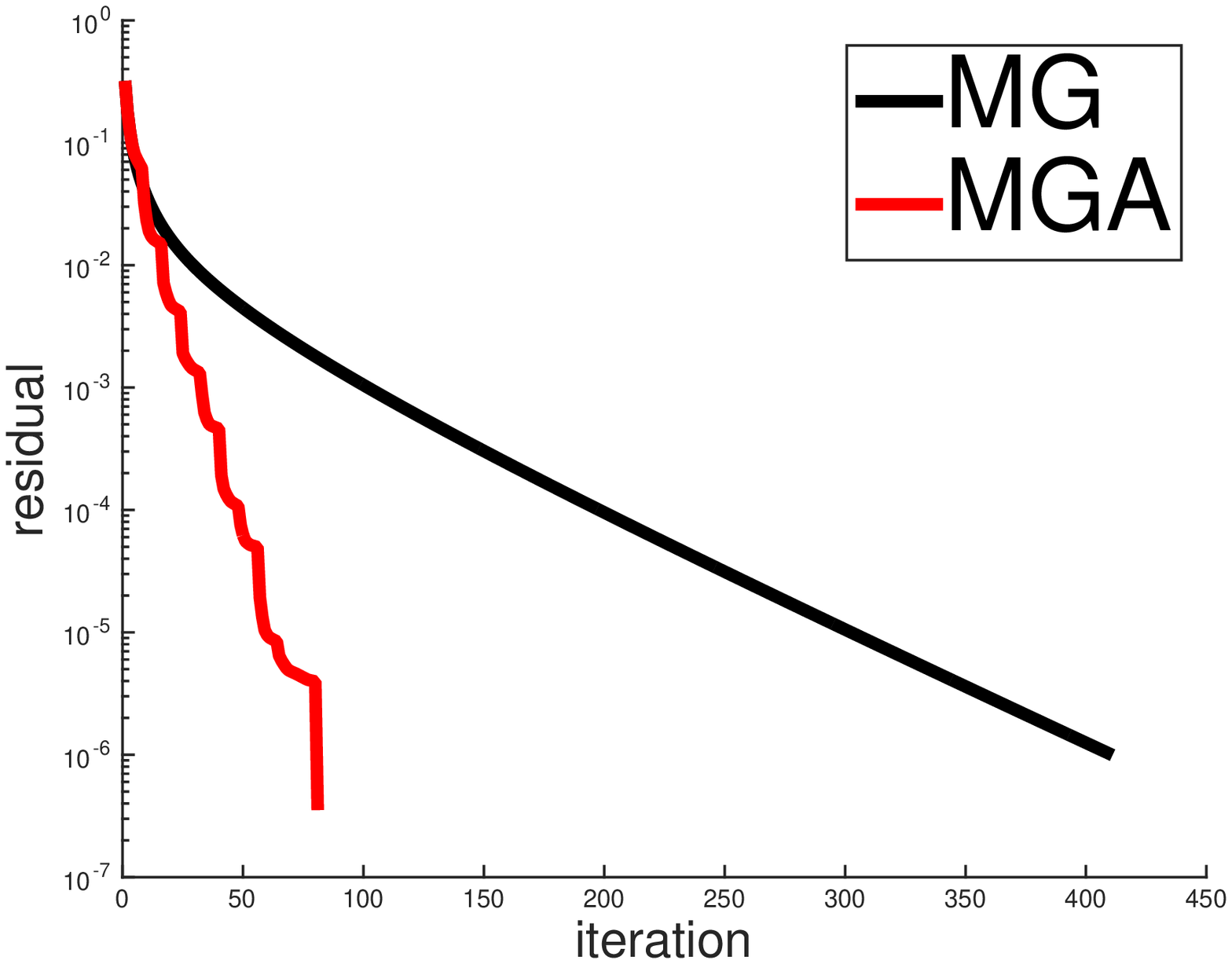}
\end{subfigure}\begin{subfigure}{.25\textwidth}
  \centering
  \includegraphics[width=1.\linewidth]{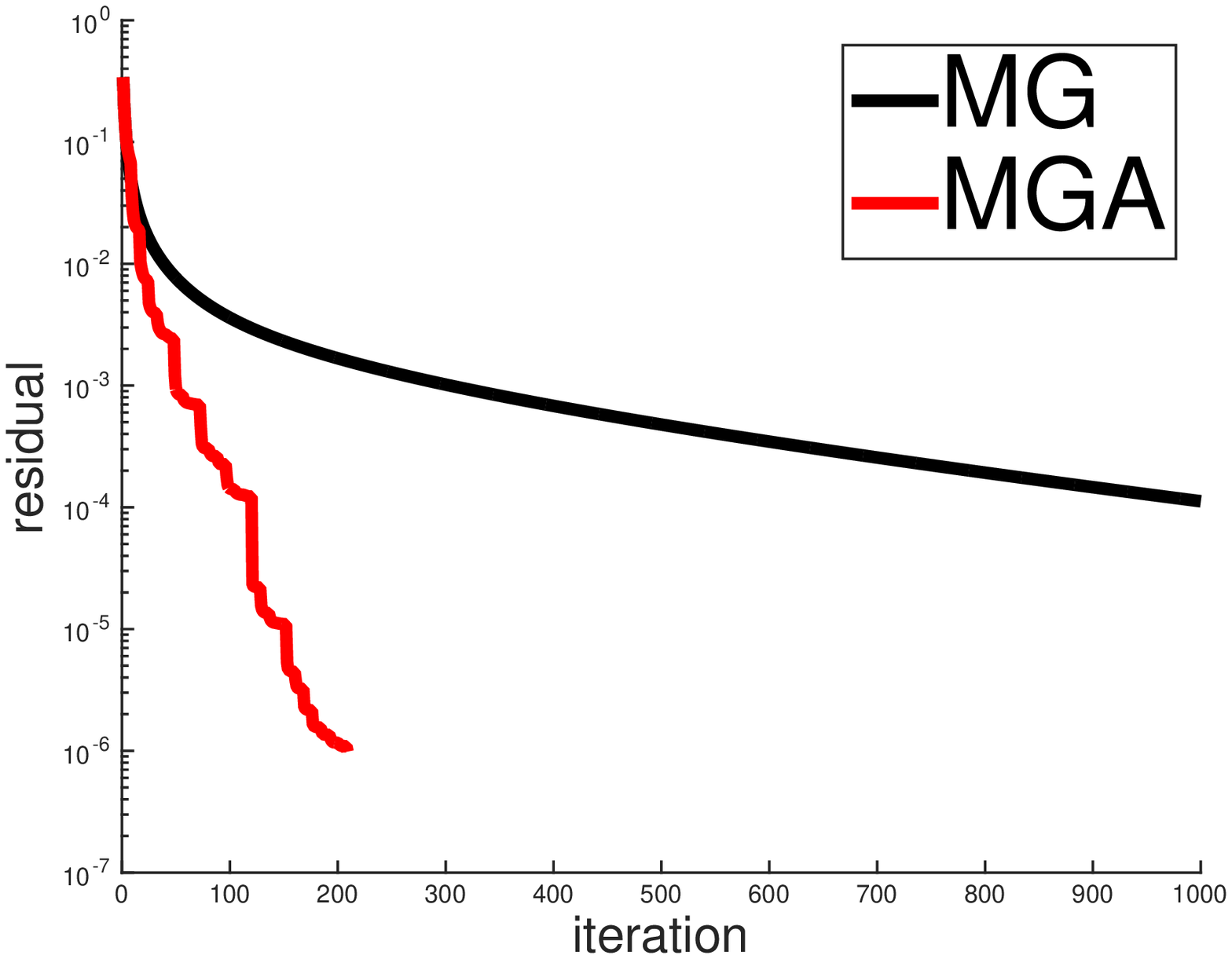}
\end{subfigure}\begin{subfigure}{.25\textwidth}
  \centering
  \includegraphics[width=1.\linewidth]{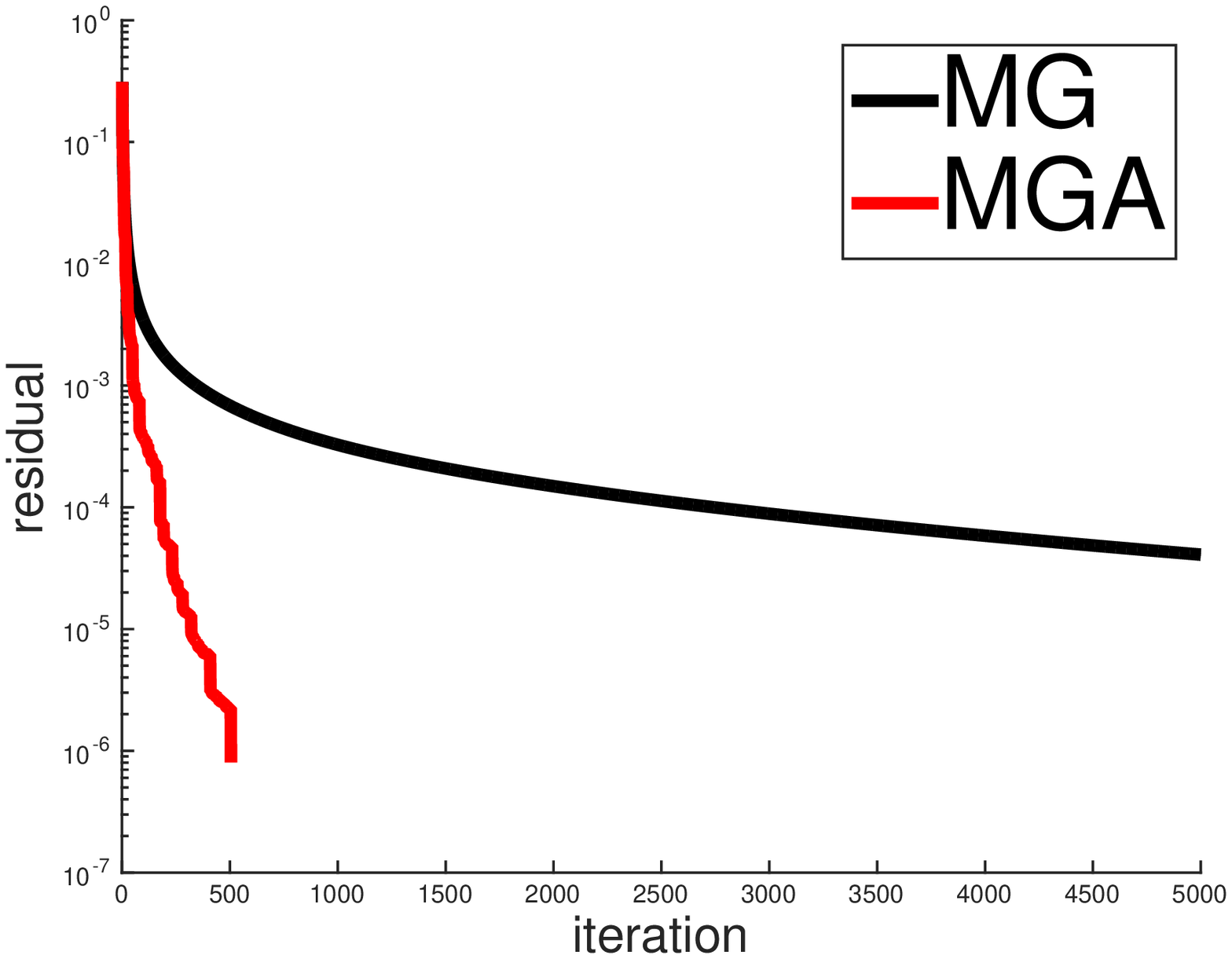}
\end{subfigure}\begin{subfigure}{.25\textwidth}
  \centering
  \includegraphics[width=1.\linewidth]{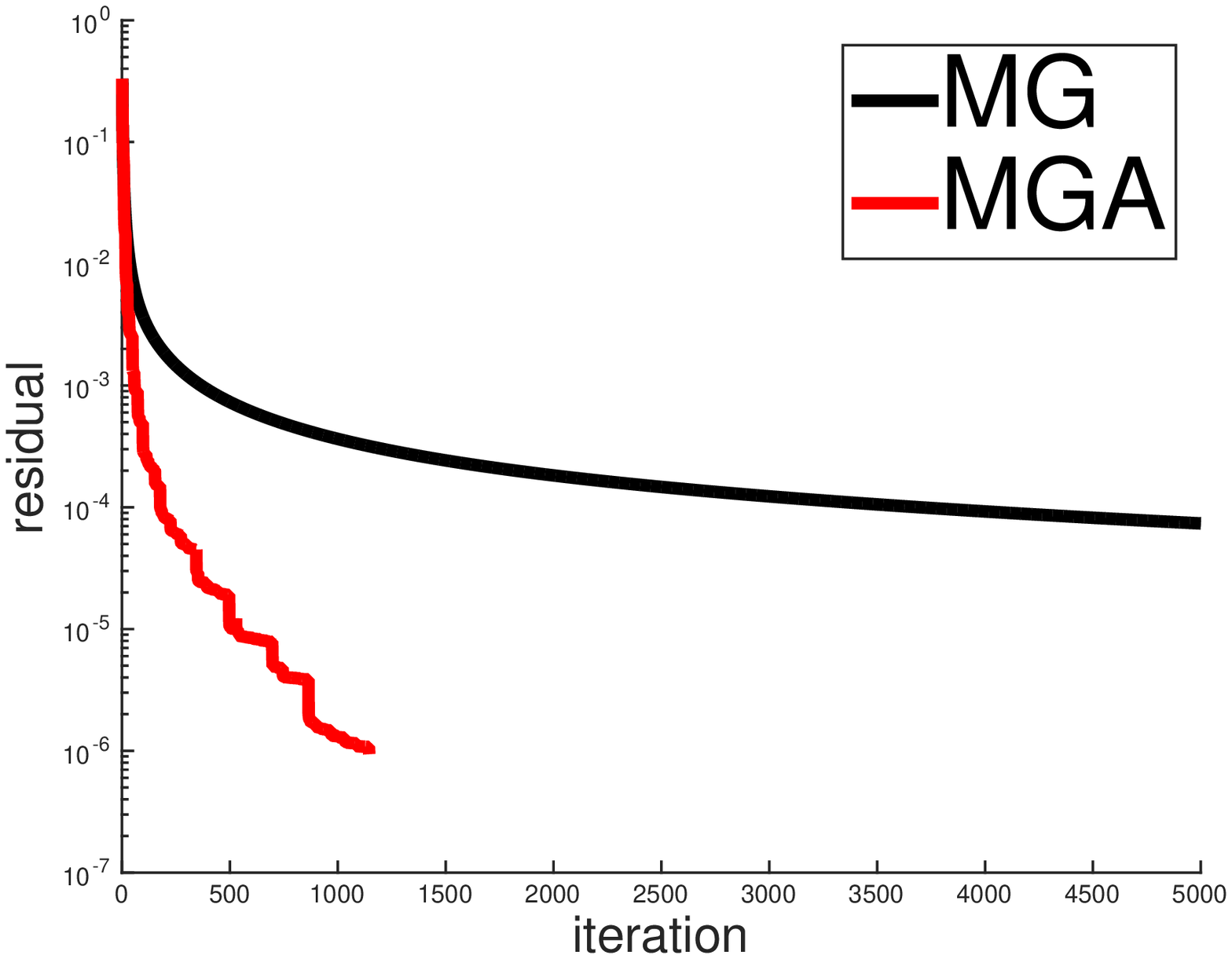}
\end{subfigure}
\caption{Top: comparison of SDA and SD. Middle: comparison of AOA and AO. Bottom: comparison of MGA and MG. Random problems are generated with $N=1000$: $\kappa=10^2$ (first), $\kappa=10^3$ (second), $\kappa=10^4$ (third), $\kappa=10^5$ (fourth).}
\label{fig:4}
\end{figure}
The problem size is chosen as $N=1000$.
Each comparison consists of four pairs of plot: $\kappa=10^2,\,10^3,\,10^4,\,10^5$, respectively.
The figures show that in all cases, the aligned methods terminate in relatively few iterations.
Further insight into the plots can be gained by observing the oscillating behavior, which reveals that SDA usually has large magnitude of oscillation, while MGA is the smoothest one.
It is known that the oscillation of a convergence curve is closely related to the numerical stability~\cite{Lamotte2002}.
In view of the convergence performance and the stability behavior for the three aligned methods, the use of the MGA step is more recommended than the SDA step.

The next experiment is a two-point boundary value problem~\cite{Friedlander1999,Dai2003b}.
The tridiagonal matrix $A$ after discretization by the finite difference method is of the form $A = \text{tridiag}(-1/h^2,\,2/h^2,\,-1/h^2)$
with $h=11/N$.
Notice that with the augmentation of matrix dimension $N$, the condition number $\kappa$ will also increase.
The purpose of this is to confirm the previous results obtained for the new methods.
Since SDA and MGA are as expected less efficient than SDC and MGC, we shall not address them again and focus on other three methods.
The AOA curve is retained for the sake of comparison.
Here, we provide results of the cases $N=10^2,\,10^3,\,10^4,\,10^5$ and illustrate the residual curves.
Fig.~\ref{fig:5} shows that MGC are quite competitive with SDC, while AOA can not beat them in all cases.
\begin{figure}[!t]
\centering
\begin{subfigure}{.5\textwidth}
  \centering
  \includegraphics[width=1.\linewidth]{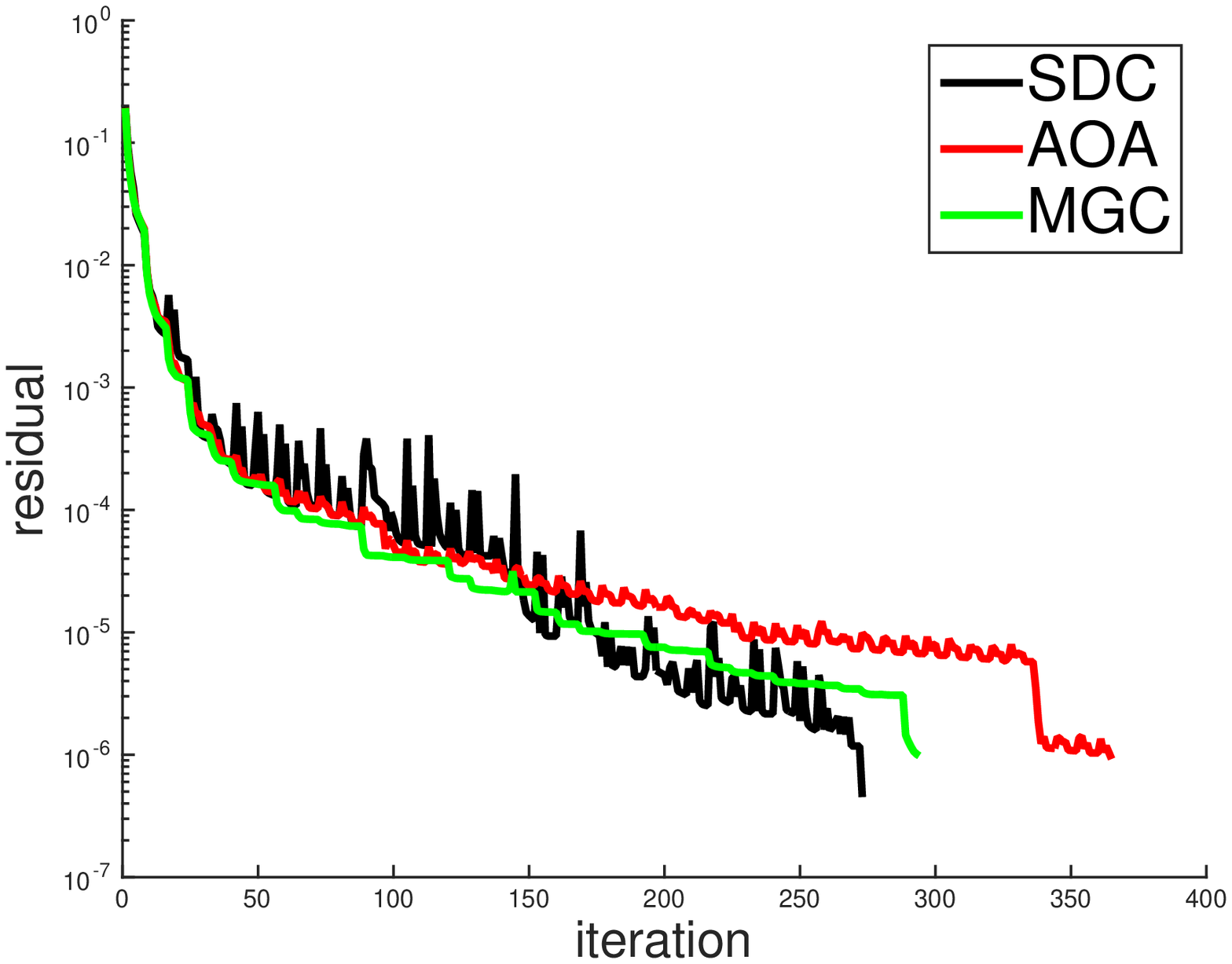}
\end{subfigure}\begin{subfigure}{.5\textwidth}
  \centering
  \includegraphics[width=1.\linewidth]{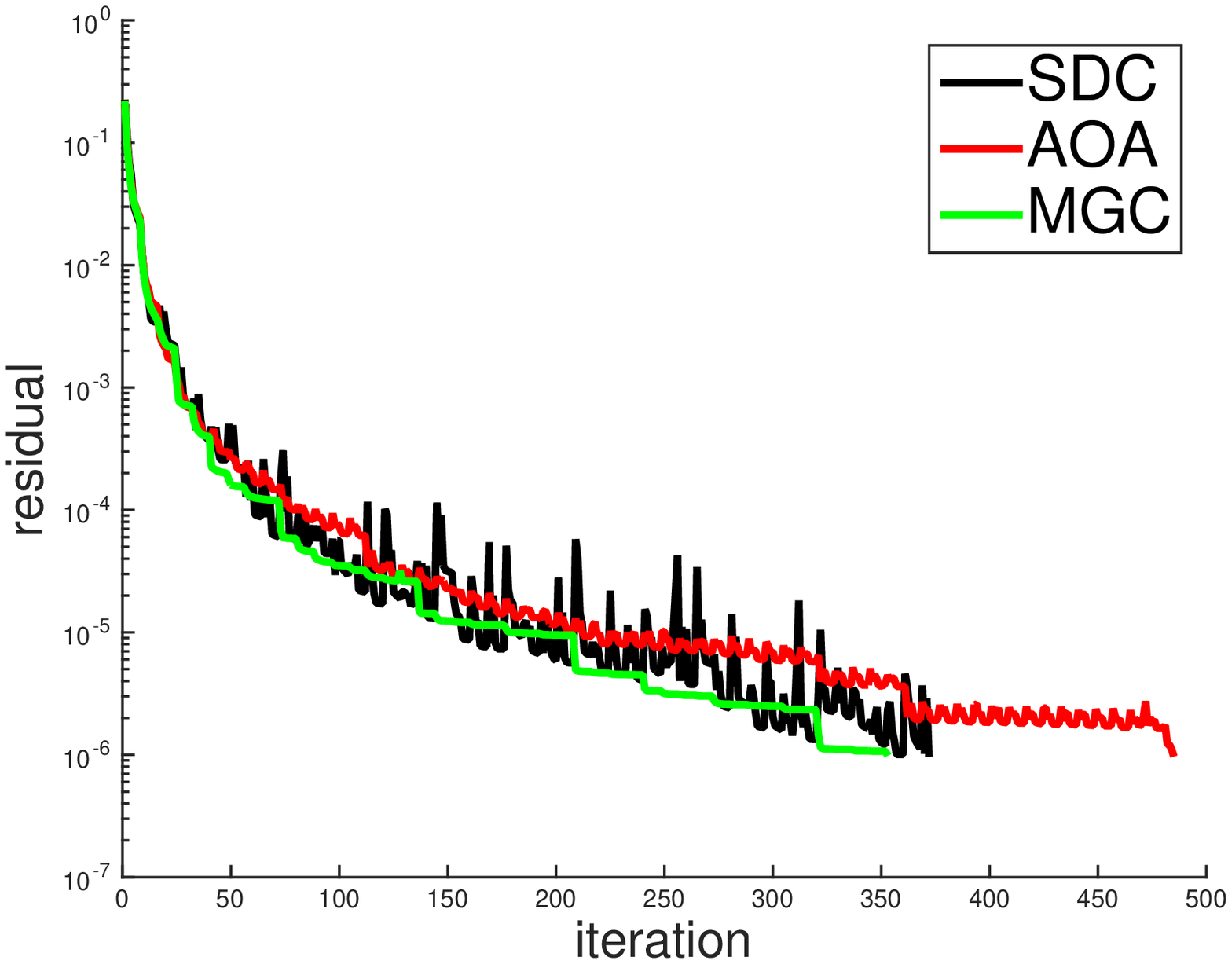}
\end{subfigure}
\begin{subfigure}{.5\textwidth}
  \centering
  \includegraphics[width=1.\linewidth]{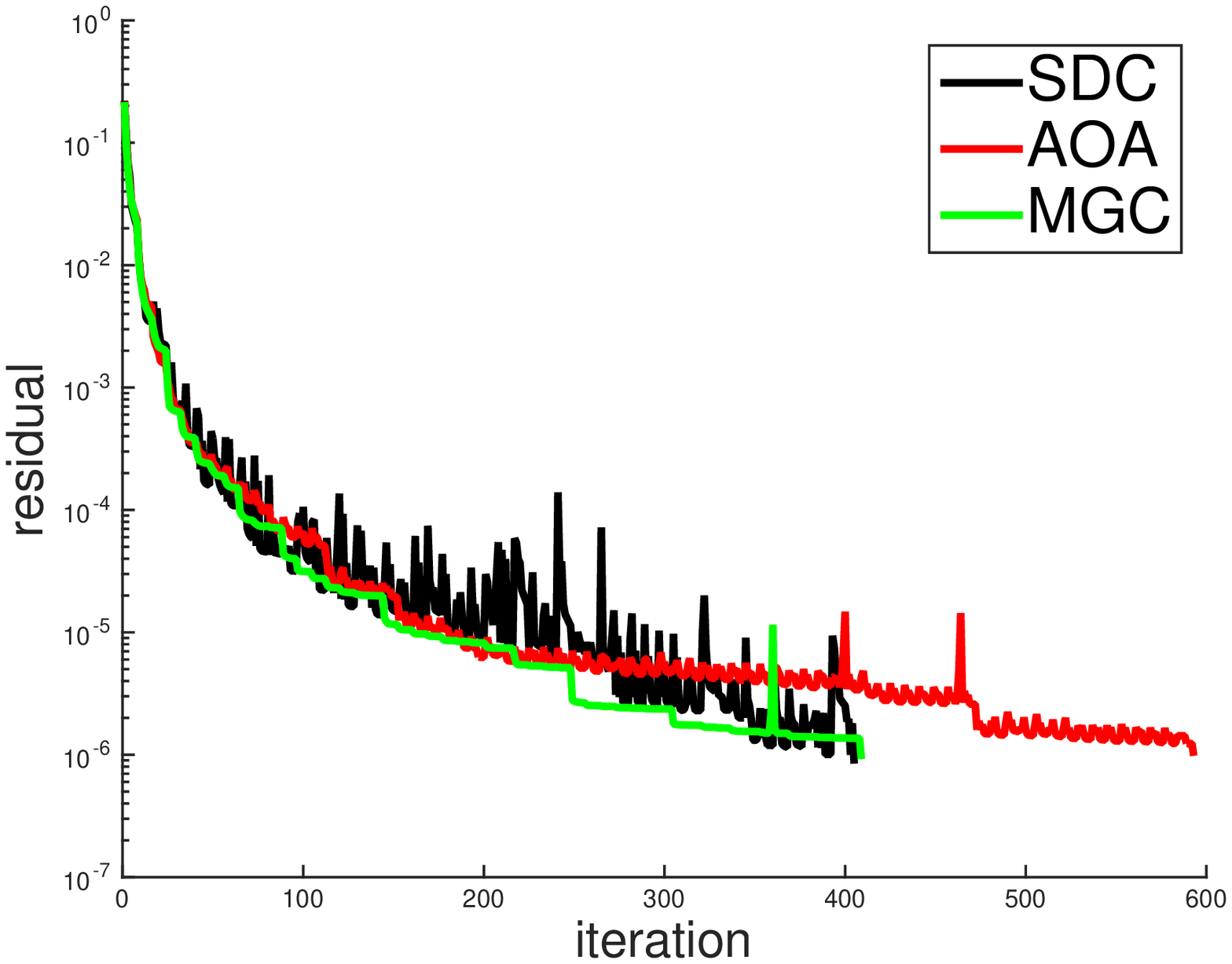}
\end{subfigure}\begin{subfigure}{.5\textwidth}
  \centering
  \includegraphics[width=1.\linewidth]{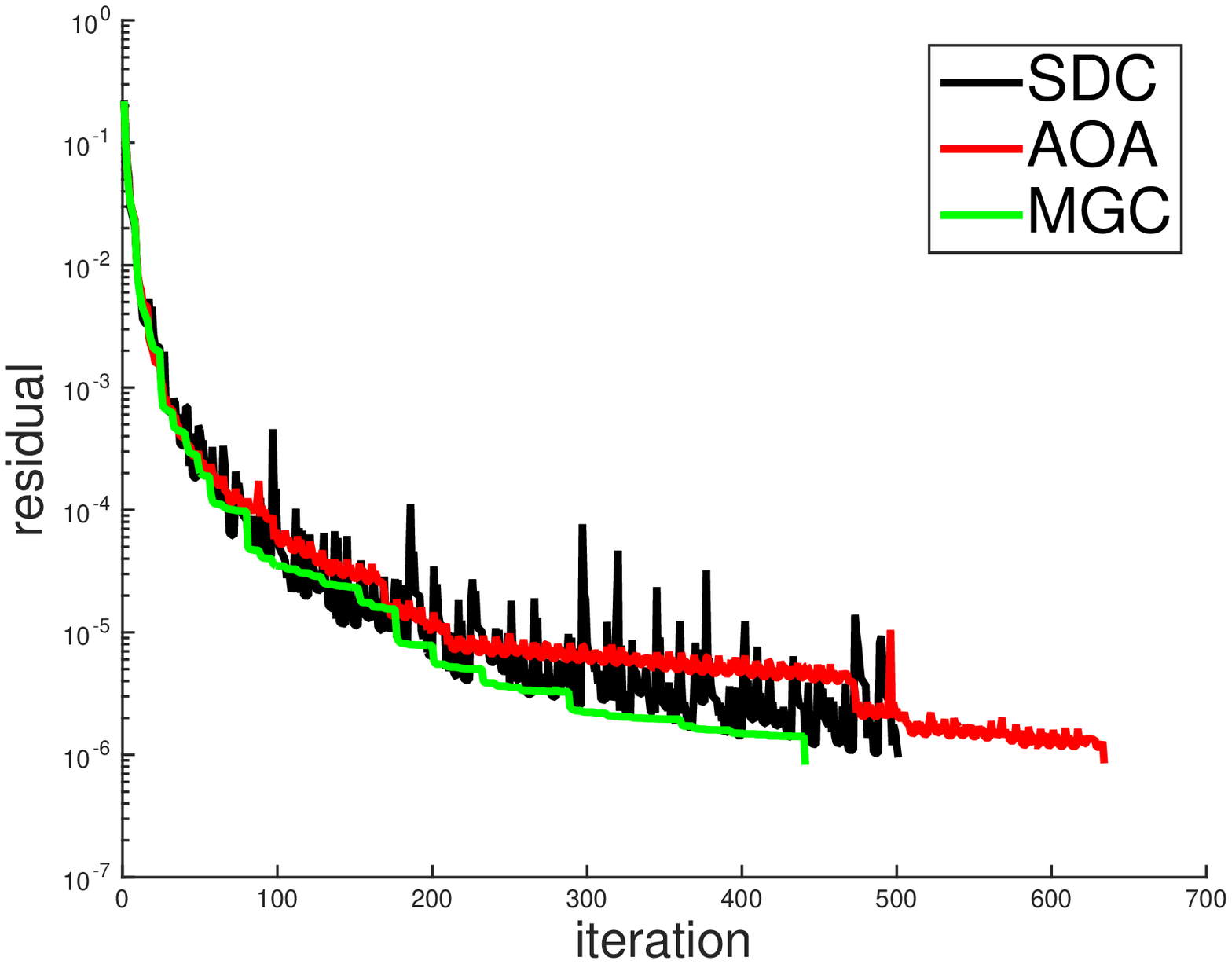}
\end{subfigure}
\caption{Comparison of different gradient methods through the two-point boundary value problems: $N=10^2$ (top-left), $N=10^3$ (top-right), $N=10^4$ (bottom-left), $N=10^5$ (bottom-right).}
\label{fig:5}
\end{figure}

Similar to the previous results, we can see that SDC oscillates mightily in all cases, while AOA is slightly better than SDC emerging from the fact that it yields smoother transitional curves between the spikes.
MGC shows the most promising performance since it gives not only a competitive convergence speed, but also a much smoother curve than other methods.
In one direction, the MG-based method minimizes indeed the residual value.
On the other hand, it is known that stability generally favors short steplengths.
Along with~\eqref{eq:rls}, the desired conclusion follows.

Finally, we compare our new methods with the conjugate gradient (CG) method~\cite{Hestenes1952}.
Two examples are used to show the robustness and efficiency of the proposed methods.
The first example concerns the random problems with perturbation generated by MATLAB, which have the following form
\[
\tilde{A}x=b,\quad \tilde{A} = A+\delta V,
\]
where $\delta$ is a small positive value and $V$ is a nonsymmetric matrix.
Still, let $\kappa$ be the condition number of $A$.
We choose $\delta=10^{-4}$.
$V$ is generated by the MATLAB function~\texttt{sprand}.
We compare also our methods with the generalized minimum residual (GMRES) method~\cite{Saad1986} in view of the perturbation.
Here we use the restarted GMRES where algorithm is restarted every $l$ iterations.
The computational results are shown in Fig.~\ref{fig:6}.
\begin{figure}[!t]
\centering
\begin{subfigure}{.5\textwidth}
  \centering
  \includegraphics[width=1.\linewidth]{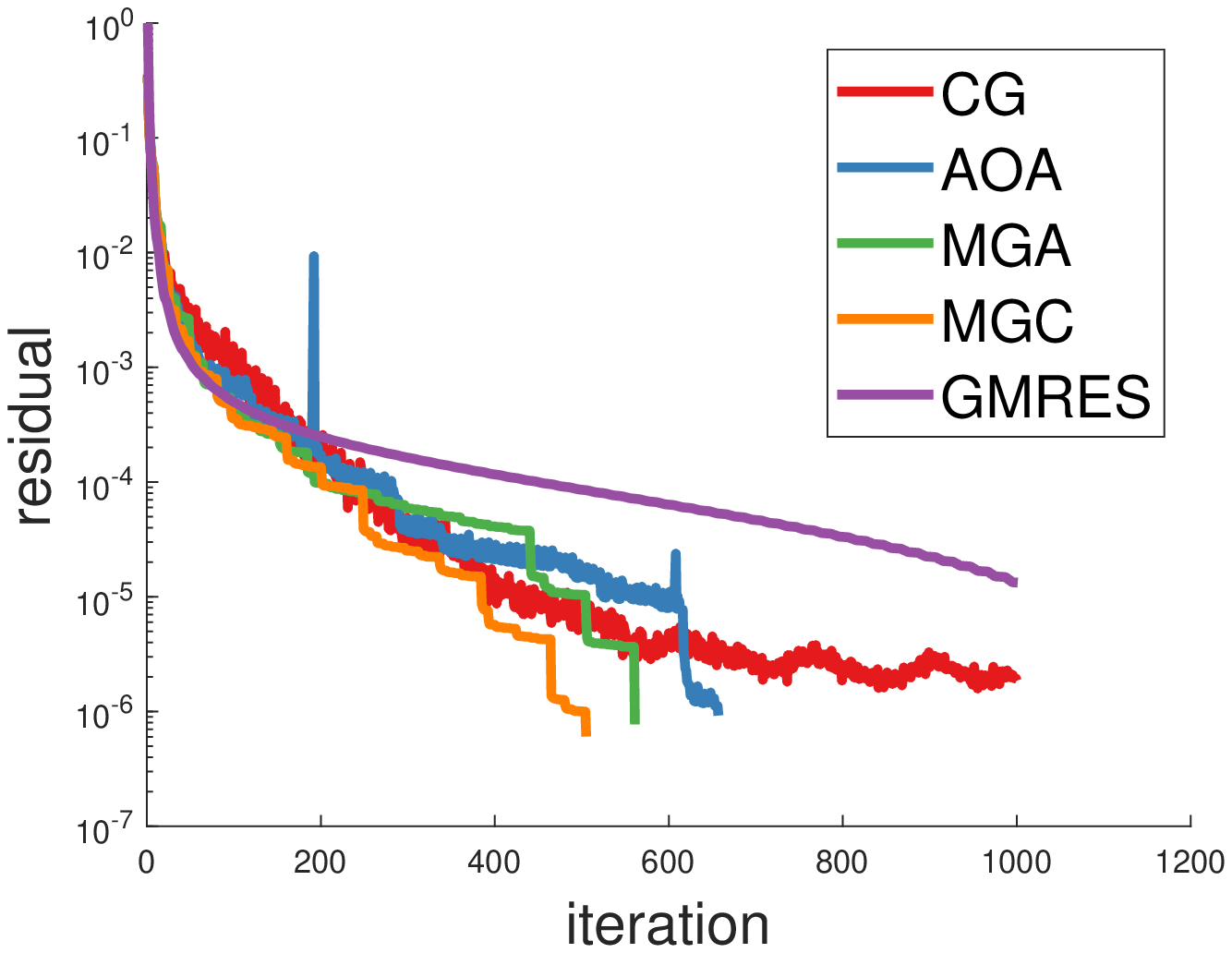}
\end{subfigure}\begin{subfigure}{.5\textwidth}
  \centering
  \includegraphics[width=1.\linewidth]{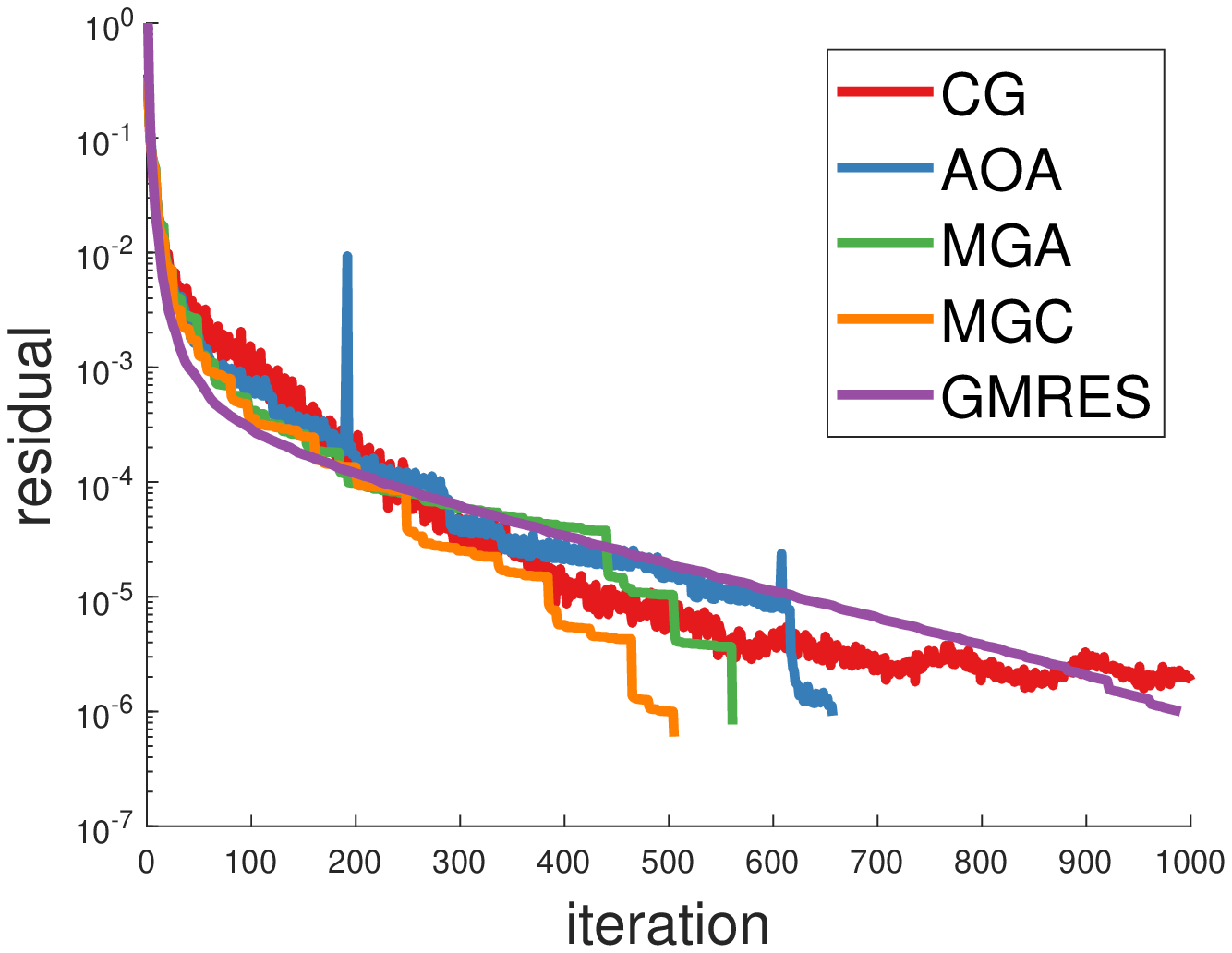}
\end{subfigure}
\begin{subfigure}{1.\textwidth}
  \centering
  \includegraphics[width=.5\linewidth]{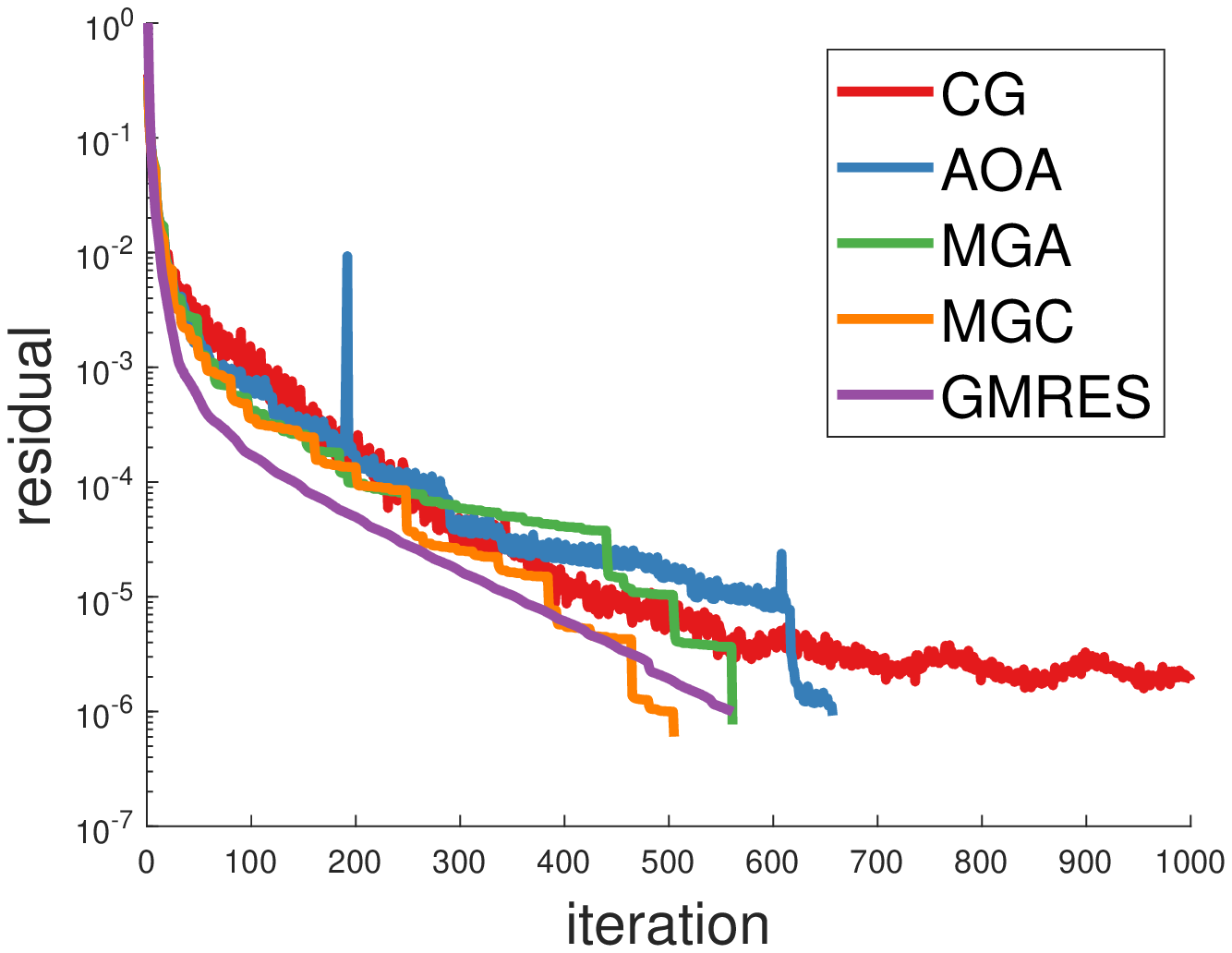}
\end{subfigure}
\caption{Comparison of the new methods with CG and restarted GMRES through random problems with perturbation where $N=10^2$ and $\kappa=10^4$. GMRES is restarted every $l$ iterations: $l=10$ (top-left), $l=20$ (top-right), $l=30$ (bottom).}
\label{fig:6}
\end{figure}
We observe that CG curve decreases in the beginning but stagnates in the end, while our new methods are robust and resistant to perturbation.
On the other hand, GMRES needs to store $l$ more vectors, which means $lN$ storage locations, and requires about $l$ more vector updates and dot products than gradient methods.
The second example is drawn from the University of Florida Sparse Matrix Collection~\cite{Davis2011} which is a large-scale system with $N=1564794$ and $\kappa=1.225\times 10^8$.
The matrix name is \texttt{Flan\_1565} with ID $2544$.
This is obtained from a 3D mechanical problem discretized by hexahedral finite elements.
The computational result is shown in Fig.~\ref{fig:7}.
\begin{figure}[!t]
\centering
\includegraphics[width=.75\linewidth]{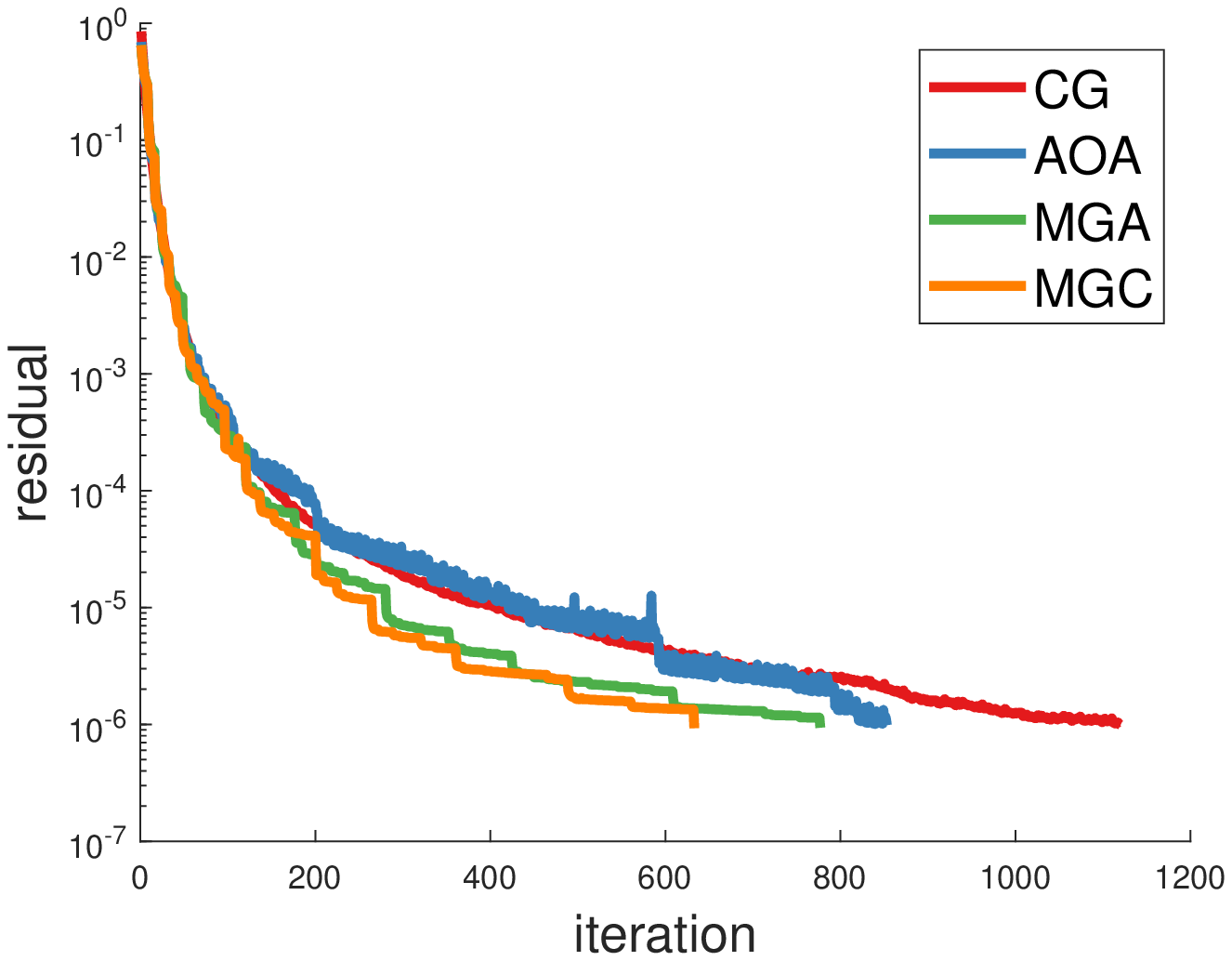}
\caption{Comparison of the new methods with conjugate gradient through a large-scale problem: $N=1564794$.}
\label{fig:7}
\end{figure}
The new methods perform better than CG in this case and the best performance is realized by MGC.

\section{Concluding remarks}
\label{sec:6}

We address first the spectral properties of the MG method.
Our analysis effectively extends that in~\cite{Nocedal2002} which includes only the SD method.
In fact, it is possible to further extend the current results based on the $P$-gradient framework as mentioned in Section~\ref{sec:2}.
We introduce here only the MG-based properties since it is the most promising candidate for a further formulation.
Additionally, our analysis shows that the Cauchy step is not an indispensable component to trigger the alignment behavior.
The Cauchy-short framework proposed in~\cite{Gonzaga2016} could thus be updated and generalized to our cases.

In this paper, we propose three new gradient methods with alignment, called AOA, MGA and MGC, respectively.
MGC shows great competitiveness to SDC, while SDA, AOA and MGA have been proved to be less efficient than other methods in most cases.
A closer examination of AOA and MGC reveals that they are more stable than SDC.
Such feature may contribute to the problem of loss of precision~\cite{Lamotte2002}.
The new methods with alignment present several advantages over the Krylov subspace methods.

There exist two main heuristics to accelerate the gradient methods.
One is to reveal the spectral property, which yields eventually the alignment methods; the other depends on the ``decreasing together'' behavior as presented in~\cite{Dai2005c}.
For example, BB and DY both possess the second feature.
According to our experiments, the former seems to be more effective than the latter.
Further investigation of different heuristics seems to be a good research topic in the future.

\section*{Acknowledgements}
This work was supported by the French national programme LEFE/INSU and the project ADOM (M\'ethodes
de d\'ecomposition de domaine asynchrones) of the French National Research Agency (ANR).

\bibliography{ref}
\bibliographystyle{abbrv}

\end{document}